\documentclass[10pt,a4paper]{article}
\usepackage[utf8]{inputenc}
\usepackage{amsmath, amsthm}
\usepackage{esint}
\usepackage{amsfonts}
\usepackage{amssymb}
\usepackage{hyperref}
\usepackage{graphicx}
\usepackage[T1]{fontenc}
\usepackage[francais,english]{babel}
\usepackage{listings}
\usepackage{vmargin} 
\usepackage{dsfont}
\usepackage{color}
\numberwithin{equation}{section}
\setmarginsrb{2cm}{2.5cm}{2cm}{2.5cm}{0cm}{0cm}{0cm}{1cm} 

\newtheorem{de}{Definition}[section]
\newtheorem{thm}[de]{Theorem}
\newtheorem{cor}[de]{Corollary}
\newtheorem{prop}[de]{Proposition}
\newtheorem{lem}[de]{Lemma}
\newtheorem{rem}[de]{Remark}

\newcommand\E{\mathcal{E}}
\newcommand\F{\mathcal{F}}
\newcommand\V{\mathcal{V}}
\newcommand\D{\mathcal{D}}
\newcommand\G{\mathcal{G}}
\newcommand\h{\mathcal{H}}
\newcommand\s{\mathfrak{S}}
\newcommand\K{\mathfrak{K}}

\newcommand\N{\mathbb{N}}

\newcommand\grho{\boldsymbol{\rho}}

\renewcommand{\textbf}[1]{\begingroup\bfseries\mathversion{bold}#1\endgroup}

\DeclareMathOperator*{\dive}{div}

\DeclareMathOperator*{\Rn}{\mathbb{R}^\textit{n}}
\DeclareMathOperator*{\R}{\mathbb{R}}

\title{On some non linear evolution systems which are perturbations of Wasserstein gradient flows}

\author {Maxime Laborde \thanks{\scriptsize CEREMADE, UMR CNRS 7534, Universit\'e Paris IX
Dauphine, Pl. de Lattre de Tassigny, 75775 Paris Cedex 16, FRANCE
\texttt{laborde@ceremade.dauphine.fr}.}}
\begin{document}

\maketitle

\begin{abstract}
This paper presents existence and uniqueness results for a class of parabolic systems with non linear diffusion and nonlocal interaction. These systems can be viewed as regular perturbations of Wasserstein gradient flows. Here we extend results known in the periodic case (\cite{CL}) to the whole space and on a smooth bounded domain. Existence is obtained using a semi-implicit Jordan-Kinderlehrer-Otto scheme and uniqueness follows from a displacement convexity argument. 

\end{abstract}
\section{Introduction.}

In this paper we study existence and uniqueness of solutions for systems of the form
\begin{eqnarray}
\label{systeme}
 \left\{\begin{array}{ll}
 \partial_t \rho_i -  \dive(\rho_i \nabla (V_i[\boldsymbol{\rho}])) + \alpha_i \dive(\rho_i\nabla F'_{i}(\rho_i)) =0 & \text{   on } \mathbb{R}^+ \times \Omega,\\

\rho_i(0, \cdot)= \rho_{i,0} &\text{   on } \Omega,\\
\end{array}\right.
\end{eqnarray}
where $i \in [\![1,l]\!]$ ($l \in \mathbb{N}^*$), $\Omega= \Rn$ or is a bounded set of $\Rn$ and $\grho:=(\rho_1,\dots,\rho_l)$ is a collection of densities. Our motivation for this system comes from its appearance in modeling interacting species. 

\smallskip
In the case of $\nabla (V_i[\boldsymbol{\rho}])=0$ or $V_i[\grho]$ does not depend of $\grho$, this system can be seen as a gradient flow in the product Wasserstein space i.e $\nabla F'_i(\rho_i)$ can be seen as the first variation of a functional $\F_i$ defined on measures. This theory started with the work of Jordan, Kinderlehrer and Otto in \cite{JKO} where they discovered that the Fokker-Planck equation can be seen as the gradient flow of $\int_{\Rn} \rho\log\rho + \int_{\Rn} V\rho$. The method that they used to prove this result is often called JKO scheme. Now, it is well-known that the gradient flow method permits to prove the existence of solution under very weak assumptions on the initial condition for several evolution equations, such as the heat equation \cite{JKO}, the porous media equation \cite{O}, degenerate parabolic equations \cite{A}, Keller-Segel equation \cite{BCC}. The general theory of gradient flow has been very much developed and is detailed in the book of Ambrosio, Gigli and Savaré, \cite{AGS}, which is the main reference in this domain.

\smallskip
However, this method is very restrictive if we want to treat the case of systems with several interaction potentials. Indeed, Di Francesco and Fagioli show in the first part of \cite{DFF} that we have to take the same (or proportional) interaction potentials, of the form $V[\rho]=W\ast \rho$ for all densities. They prove an existence/uniqueness result of \eqref{systeme} using gradient flow theory in a product Wasserstein space without diffusion ($\alpha_i=0$) and with $l=2$, $V_1[\rho_1,\rho_2]:=W_{1,1} \ast \rho_1 + W_{1,2} \ast \rho_2$ and $V_2[\rho_1,\rho_2]:=W_{2,2} \ast \rho_2 + W_{2,1} \ast \rho_1$ where $ W_{1,2}$ and $W_{2,1}$ are equals or proportionals. Nevertheless in the second part of \cite{DFF}, they introduce a new semi-implicit JKO scheme to treat the case where $ W_{1,2}$ and $W_{2,1}$ are not proportional. In other words, they use the usual JKO scheme freezing the measure in $V_i[\grho]$. 

\smallskip
The purpose of this paper is to add a nonlinear diffusion in the system studied in \cite{DFF}. Unfortunately, this term requires strong convergence to pass to the limit. This can be obtained using an extension of Aubin-Lions lemma proved by Rossi and Savaré in \cite{RS} and recalled in theorem \ref{Rossi-Savaré}. This theorem requires separately time-compactness and space compactness to obtain a strong convergence in $L^m((0,T)\times \Omega)$. The time-compactness follows from classical estimate on the Wasserstein distance in the JKO scheme. The difficulty is to find the space-compactness. This problem has already been solved in \cite{CL} in the periodic case using the semi-implict scheme of \cite{DFF}. In this paper we extend this result on the whole space $\Rn$ or on a smooth bounded domain. On the one hand in $\Rn$, we will use the same argument than in \cite{CL}. We use the powerful flow interchange argument of Matthes, McCann and Savaré \cite{MMCS} and also used in the work of Di Francesco and Matthes \cite{DFM}. The differences with the periodic case are that functionals are not, a priori, bounded from below and we can not use Sobolev compactness embedding theorem. On the other hand in a bounded domain, the flow interchange argument is very restrictive because it forces us to work in a convex domain and to impose some boundary condition on $V_i[\grho]$. To avoid these assumptions, we establish a $BV$ estimate to obtain compactness in space and then to find the strong convergence needed.

\smallskip

The paper is composed of seven sections. In section \ref{main result}, we start to recall some facts on the Wasserstein space and we state our main result, theorem \ref{existence Rn}. Sections \ref{si JKO}, \ref{flow interchange} and \ref{limit} are devoted to prove theorem \ref{existence Rn}. In section \ref{si JKO}, we introduce a semi-implict JKO scheme, as in \cite{DFF}, and resulting standard estimates. Then, in section \ref{flow interchange}, we recall the flow interchange theory developed in \cite{MMCS} and we find a stronger estimate on the solution's gradient, which can be done by differentiating the energy along the heat flow. In section \ref{limit}, we establish convergence results and we prove theorem \ref{existence Rn}. Section \ref{borné} deals with the case of a bounded domain. In the final section \ref{uniqueness}, we show uniqueness of \eqref{systeme} using a displacement convexity argument. 

\section{Wasserstein space and main result.}
\label{main result}
Before stating the main theorem, we recall some facts on the Wasserstein distance.
\paragraph{The Wasserstein distance.}

We introduce $$ \mathcal{P}_{2}(\Rn) :=\left\{ \mu \in \mathcal{M}(\Rn ; \mathbb{R}^+) \, : \, \int_{\Rn} \,d\mu =1 \text{ and } M(\mu):= \int_{\Rn} |x|^2 d \mu(x) < +\infty \right\},$$
and we note $\mathcal{P}_{2}^{ac}(\Rn)$ the subset of $\mathcal{P}_{2}(\Rn)$ of probability measures on $\Rn$ absolutely continuous with respect to the Lebesgue measure. The Wasserstein distance of order $2$, $W_2(\rho,\mu)$, between $\rho$ and $\mu$ in $\mathcal{P}_{2}(\Rn)$, is defined by
$$ W_2(\rho,\mu):= \inf_{\gamma \in \Pi(\rho,\mu)} \left( \int_{\Rn \times \Rn} |x-y|^2 \,d\gamma(x,y) \right)^{1/2},$$
where $\Pi(\rho,\mu)$ is the set of probability measures on $\Rn\times \Rn$ whose first marginal is $\rho$ and second marginal is $\mu$. It is well known that $\mathcal{P}_{2}(\Rn)$ equipped with $W_2$ defines a metric space (see for example \cite{S,V1,V2}). Moreover if $\rho \in \mathcal{P}_{2}^{ac}(\Rn)$ then $W_2(\rho, \mu)$ admits a unique optimal transport plan $\gamma_T$ and this plan is induced by a transport map, i.e $\gamma_T=(Id\times T)_{\#}\rho$, where $T$ is the gradient of a convex function (see \cite{B}).\\
Now if $\grho,\boldsymbol{\mu} \in \mathcal{P}_{2}(\Rn)^l$, we define the product distance by
$$ W_2(\grho,\boldsymbol{\mu}) := \sum_{i=1}^l W_2(\rho_i,\mu_i),$$ 
or every equivalent metric as $W_2^2(\grho,\boldsymbol{\mu}) := \sum_{i=1}^l W_2^2(\rho_i,\mu_i)$.
We can define also the $1$-Wasserstein distance by
$$ W_1(\rho,\mu):= \inf_{\gamma \in \Pi(\rho,\mu)} \left( \int_{\Rn \times \Rn} |x-y| \,d\gamma(x,y) \right),$$
and the Kantorovich duality formula (see \cite{V1,V2,S}) gives 
$$ W_1(\rho,\mu)=\sup\left\{ \int_{\Rn} \varphi \,d(\rho-\mu)\; : \; \varphi \in L^1(d|\rho-\mu|)\cap Lip_1(\Rn) \right\},$$
with $Lip_1$ is the set of $1$-Lipschitz continuous functions. Then for all $\rho,\mu \in  \mathcal{P}_{2}(\Rn)$ and $\varphi \in Lip(\Rn)$ we have
\begin{eqnarray}
\label{Lipschitz W2}
 \int_{\Rn} \varphi \,d(\rho-\mu) \leqslant CW_1(\rho,\mu) \leqslant CW_2(\rho,\mu).
\end{eqnarray}

\paragraph{Main result.}
Let $l \in \mathbb{N}^*$ and for all $i \in [\![ 1 ,l ]\!]$, we define $V_i \, :\, \mathcal{P}(\Rn)^l \rightarrow \mathcal{C}^{2}(\Rn)$ continuous such that:\\

\begin{itemize}
\item For all $\grho=(\rho_1, \dots, \rho_l) \in \mathcal{P}(\Rn)^l$,
\begin{align}
\label{Hyp 1 V}
V_i[\boldsymbol{\rho}] \geqslant 0,
\end{align}

\item There exists $C>0$ such that for all $\grho \in \mathcal{P}(\Rn)^l$, $x\in \Rn$,
\begin{align}
\label{Hyp 2 V}
 \|  \nabla (V_i[\grho])  \|_{L^\infty(\Rn)} + \|  D^2 (V_i[\grho])  \|_{L^\infty(\Rn)} \leqslant C ,
\end{align}
i.e $V_i[\grho]$ and $\nabla (V_i[\grho])$ are Lipschitz functions and the Lipschitz constants do not depend on the measure.\\
\item There exists $C>0$ such that for all $\boldsymbol{\nu}, \boldsymbol{\sigma} \in \mathcal{P}(\mathbb{R}^n)^l$,
\begin{align}
\label{Hyp 3 V}
\| \nabla(V_i[\boldsymbol{\nu}]) -\nabla(V_i[\boldsymbol{\sigma}]) \|_{L^\infty(\Rn)} \leqslant C  W_2 (\boldsymbol{\nu} ,\boldsymbol{\sigma}).
\end{align} 
\end{itemize} 

\begin{rem}
The assumption \eqref{Hyp 1 V} can be replaced by $V_i[\grho]$ is bounded by below for all $\grho$.
\end{rem}
\bigskip
Let $m \geqslant 1$, we define the class of functions  $\h_m$ by 
$$ \h_m:= \{x\mapsto x\log(x)\}  \text{   if } m=1,$$
and, if $m >1$, $\h_m$ is the class of strictly convex superlinear functions $F \, : \, \R^+ \rightarrow \R$ which satisfy 
\begin{align}
\label{hyp F}
F(0)=F'(0)=0, \qquad F''(x) \geqslant Cx^{m-2} \text{  and  } P(x):=xF'(x)-F(x) \leqslant C(x+x^{m}).
\end{align} 
The two first assumptions imply that if $m>1$ and $F \in \h_m$ then $F$ controls $x^m$.
\bigskip

Before giving a definition of solution of \eqref{systeme}, we recall that the nonlinear diffusion term can be rewrite as
$$ \dive(\rho \nabla F'(\rho))=\Delta P(\rho),$$
where $P(x):=xF'(x)-F(x) $ is the pressure associated to $F$.

\begin{de}
We say that $(\rho_1, \dots, \rho_l) \, : \, [0, +\infty[ \rightarrow \mathcal{P}_{2}^{ac}(\Rn)^l$ is a weak solution of \eqref{systeme} if for all $i\in [\![ 1 ,l ]\!]$, $\rho_i \in \mathcal{C}([0,T],\mathcal{P}_{2}^{ac}(\Rn))$, $P_{i}(\rho_i) \in L^1(]0,T[\times \Rn)$ for all $T<\infty$ and for all $\varphi_1, \dots, \varphi_l \in \mathcal{C}^\infty_c([0,+\infty[\times \Rn)$,
$$ \int_0^{+\infty} \int_{\Rn} \left[ \left( \partial_t \varphi_i - \nabla \varphi_i \cdot \nabla (V_i[\grho]) \right)\rho_i +\alpha_i \Delta \varphi_i P_{i}(\rho_i) \right] = -\int_{\Rn} \varphi_i(0,x)\rho_{0,i}(x).$$
\end{de}

With this definition of solution we have the following result

\begin{thm}
\label{existence Rn}
For all $i \in [\![1,l]\!]$, let  $F_i \in \h_{m_i}$, with $m_i\geqslant 1$, and $V_i$ satisfies \eqref{Hyp 1 V}, \eqref{Hyp 2 V} and \eqref{Hyp 3 V}. Let $\alpha_1,\dots,\alpha_l$ positive constants. If $\rho_{0,i} \in \mathcal{P}_{2}^{ac}(\Rn)$ satisfies\begin{align}
\label{hyp energie initiale}
\mathcal{F}_i(\rho_{0,i})+\mathcal{V}_i(\rho_{0,i}|\boldsymbol{\rho_{0}}) < +\infty,
\end{align}
with
$$ \F_i(\rho) :=\left\{ \begin{array}{ll}
\int_{\Rn} F_{i}(\rho(x))\,dx & \text{ if } \rho \ll \mathcal{L}^n,\\
+\infty & \text{ otherwise,}
\end{array}\right.   \text{  and } \V_i(\rho| \boldsymbol{\mu}):= \int_{\R} V_i[\boldsymbol{\mu}] \rho \,dx.$$
then there exist  $(\rho_1, \dots, \rho_l) \, : \, [0, +\infty[ \rightarrow \mathcal{P}_{2}^{ac}(\Rn)^l$, continuous with respect to $W_2$, weak solution of \eqref{systeme}.
\end{thm}

\begin{rem}
In the following, to simplify the proof, we take $\alpha_i =1$.
\end{rem}

\bigskip

\section{Semi-implicit JKO scheme.}\label{si JKO}
In this section, we introduce the semi-implicit JKO scheme, as \cite{DFF}, and we find the first estimates as in the usual JKO scheme.

Let $h>0$ be a time step, we construct $l$ sequences with the following iterative discrete scheme: for all $i \in [\![ 1 ,l ]\!]$, $\rho_{i,h}^0=\rho_{i,0}$ and for all $k\geqslant 1$, $\rho_{i,h}^k$ minimizes
$$ \E_{i,h}(\rho | \grho_{h}^{k-1}) := W_2^2(\rho, \rho_{i,h}^{k-1}) +2h\left( \F_i(\rho) + \V_i(\rho |\grho_{h}^{k-1}) \right),$$
on $\rho \in \mathcal{P}_{2}^{ac}(\Rn)$, with $\grho_{h}^{k-1}=(\rho_{1,h}^{k-1}, \dots, \rho_{l,h}^{k-1})$.\\

In the next proposition, we show that all these sequences are well defined. We start to prove that it is well defined for one step and after in remark \ref{step k}, we extend the result for all $k$.

\begin{prop}
\label{existence et unicité minimiseur}
Let $\grho_0=(\rho_{1,0}, \dots, \rho_{l,0}) \in \mathcal{P}_{2}^{ac}(\Rn)^l$, there exists a unique solution $\grho_h^1=(\rho_{1,h}^{1}, \dots ,\rho_{l,h}^{1}) \in \mathcal{P}_{2}^{ac}(\Rn)^l$ of the minimization problem above.
\end{prop}

\begin{proof}
First of all, we distinguish the case $m_i >1$ from $m_i=1$.
\begin{itemize}
\item If $m_i >1$, then $\E_{i,h}(\rho | \boldsymbol{\mu}) \geqslant 0,$ for all $\rho, \mu_1, \dots, \mu_l \in \mathcal{P}_{2}^{ac}(\Rn)$. Let $\rho_\nu$ be a minimizing sequence. As $\mathcal{E}_{i,h}(\rho_{i,0}|\grho_{0}) <+\infty$ (according to \eqref{hyp energie initiale}), $(\mathcal{E}_{i,h}(\rho_\nu|\grho_{0}))_\nu$ is bounded above. So there exists  $C > 0$ such that
$$0 \leqslant \mathcal{F}_i(\rho_\nu) \leqslant C \text{  and  } W_2(\rho_\nu,\rho_{i,0}) \leqslant C.$$

From the second inequality, it follows that the second moment of $\rho_\nu$ is bounded.

\item Now if $m_i=1$, following \cite{JKO}, we obtain 
\begin{eqnarray}
\label{min E}
\E_{i,h}(\rho | \boldsymbol{\mu}) \geqslant \frac{1}{4}M(\rho) -C(1+M(\rho))^\alpha - \frac{1}{2}M(\rho_{i,h}^{0}),
\end{eqnarray}
with some $0<\alpha<1$. And since $x\mapsto \frac{1}{4}x-C(1+x)^\alpha$ is bounded below, we see that $\E_{i,h}$ is bounded below.

Let $\rho_\nu$ be a minimizing sequence. Then we have $(\F_i(\rho_\nu))_\nu$ bounded above. Indeed, as $\mathcal{E}_{i,h}(\rho_{i,0}|\grho_{0}) <+\infty$, $(\mathcal{E}_{i,h}(\rho_\nu|\grho_{0}))_\nu$ is bounded above and from \eqref{Hyp 1 V} we get,
$$\int_{\R} V_i [\grho_{0}](x) \rho_\nu (x) \,dx \geqslant 0,$$
so $(\F_i(\rho_\nu))_\nu$ is bounded above. 
According to \eqref{min E}, $(M(\rho_\nu))_\nu$ is bounded. Consequently $(\F_i(\rho_\nu))_\nu$ is bounded because $\F_i(\rho) \geqslant -C(1+M(\rho))^\alpha$.
\end{itemize}

\bigskip

In both cases, using Dunford-Pettis' theorem, we deduce that there exists $\rho_{i,h}^{1} \in \mathcal{P}_{2}^{ac}(\Rn)$ such that
$$ \rho_\nu \rightharpoonup \rho_{i,h}^{1} \text{  weakly in } L^1(\Rn).$$
It remains to prove that $\rho_{i,h}^{1} $ is a solution for the minimization problem. But since $\F_i$ and $W_2^2(\cdot,\rho_{i,0})$ are weakly lower semi-continuous in $L^1(\Rn)$, we have
$$ \E_{i,h}(\rho_{i,h}^{1}|\grho_0) \leqslant \liminf_{\nu \nearrow +\infty} \E_{i,h}(\rho_\nu|\grho_0).$$

To conclude the proof, we show that the minimizer is unique. This follows from the convexity of $\V_i(\cdot |\grho_{0})$ and $\rho \in \mathcal{P}_{2}^{ac}(\Rn) \mapsto  W_2^2(\rho, \rho_{i,h}^{0})$ and the strict convexity of $\F_i$.

\end{proof}

\begin{rem}
\label{step k}
By induction, proposition \ref{existence et unicité minimiseur} is still true for all $k\geqslant 1$:\\ 
the proof is similar when we take $k-1$ instead of $0$ and if we notice that for all $i$,
$$ \F_i(\rho_{i,h}^{1}) + \V_i(\rho_{i,h}^1|\grho_{h}^1) \leqslant \F_i(\rho_{i,0}) + \V_i(\rho_{i,0}|\grho_0)+ C W_2( \grho_{0},\grho_{h}^{1}) \leqslant C.$$
The last inequality is obtained from the minimization scheme and from the assumptions \eqref{Hyp 1 V}, \eqref{Hyp 3 V} and \eqref{hyp energie initiale}. By induction it becomes, for all $k\geqslant 2$,
$$ \F_i(\rho_{i,h}^{k-1}) + \V_i(\rho_{i,h}^{k-1}|\grho_{h}^{k-1}) \leqslant \F_i(\rho_{i,0}) + \V_i(\rho_{i,0}|\grho_0)+ C\sum_{j=1}^{k-1}W_2( \grho_{h}^{j-1},\grho_{h}^{j}) \leqslant C.$$
This inequality shows $\E_{i,h}(\rho_{i,h}^{k-1}|\grho_{h}^{k-1}) <+\infty$ and so we can bound $(\F_i(\rho_\nu))_\nu$ in the previous proof.
\end{rem}

\bigskip
Thus we proved that sequences $(\rho_{i,h}^k)_{k\geqslant 0}$ are well defined for all $i \in [\![ 1 ,l ]\!]$. Then we define the interpolation $\rho_{i,h} : \mathbb{R}^+ \rightarrow \mathcal{P}_{2}^{ac}(\Rn)$ by, for all $k \in \mathbb{N}$,
\begin{equation}
\label{interpolation}
\rho_{i,h}(t)= \rho_{i,h}^k \text{ if } t\in ((k-1)h,kh].
\end{equation}

The following proposition shows that $\rho_{i,h}$ are solutions of a discrete approximation of the system \eqref{systeme}.

\begin{prop}
\label{equation discrete}
Let $h>0$, for all $T>0$, let $N$ such that $Nh=T$ and for all $(\phi_1, \dots, \phi_l) \in \mathcal{C}^\infty_c ([0,T)\times \Rn)^l$, then
\begin{eqnarray*}
\int_0^T \int_{\Rn} \rho_{i,h}(t,x)  \partial_t \phi_i(t,x) \,dxdt&=&-h\sum_{k=0}^{N-1}\int_{\Rn} P_i(\rho_{i,h}^{k+1}(x))  \Delta \phi_i (t_{k},x)\,dx\\
&+&h\sum_{k=0}^{N-1}\int_{\Rn} \nabla (V_i[\grho_h^{k}]) \cdot  \nabla \phi_i(t_{k},x) \rho_{i,h}^{k+1}(x)\,dx\\
&+&\sum_{k=0}^{N-1}\int_{\Rn \times \Rn} \mathcal{R}[\phi_i(t_{k},\cdot)](x,y) d\gamma_{i,h}^k (x,y)\\
&-& \int_{\Rn} \rho_{i,0}(x) \phi_i(0,x) \, dx,
\end{eqnarray*}
with, for all $\phi \in \mathcal{C}^\infty_c([0,T) \times \Rn)$,
$$ |\mathcal{R}[\phi](x,y)| \leqslant \frac{1}{2} \|D^2 \phi \|_{L^\infty ([0,T) \times \Rn)} |x- y|^2,$$
and $\gamma_{i,h}^k$ is the optimal transport plan in $\Gamma (\rho_{i,h}^{k},\rho_{i,h}^{k+1})$.

\end{prop}

\begin{proof}
We split the proof in two steps. We first compute the first variation of $\E_{i,h}(\cdot |\grho_{h}^{k})$ and then we integrate in time. In the following, $i$ is fixed in $[\![1,l]\!]$.\\

\begin{itemize}
\item {\it First step:} For all $k\geqslant 0$, if $\gamma_{i,h}^k$ is the optimal transport plan in $\Gamma (\rho_{i,h}^{k},\rho_{i,h}^{k+1})$ then
\begin{eqnarray*}
\int_{\Rn} \varphi_i(x) (\rho_{i,h}^{k+1}(x)-\rho_{i,h}^{k}(x))&=&h\int_{\Rn} P_i(\rho_{i,h}^{k+1}(x))  \Delta \varphi_i(x) \, dx\\
&-&h\int_{\Rn} \nabla(V_i [\grho_{h}^{k}])(x) \cdot \nabla \varphi_i (x) \rho_{i,h}^{k+1}(x) \,dx\\
&-& \int_{\Rn\times \Rn} \mathcal{R}[\varphi_i](x,y) d\gamma_{i,h}^k (x,y),
\end{eqnarray*}
for all $\varphi_i \in \mathcal{C}^\infty _c (\Rn)$ and for all $i  \in [\![ 1 ,l ]\!]$.

\bigskip

To obtain this equality, we compute the first variation of $\E_{i,h}(\cdot |\grho_{h}^{k} )$. Let $\xi_i \in  \mathcal{C}^\infty _c (\Rn,\Rn)$ and $\tau >0$ and let $\Psi_\tau$ defined by
$$ \partial_\tau \Psi_\tau = \xi_i \circ \Psi_\tau, \qquad \Psi_0=Id.$$
After we perturb $\rho_{i,h}^{k+1}$ by $\rho_\tau=(\Psi_\tau)_\sharp \rho_{i,h}^{k+1}$. According to the definition of $\rho_{i,h}^{k+1}$, we get
\begin{eqnarray}
\label{premiere varaiation}
 \frac{1}{\tau}\left(\E_{i,h}(\rho_\tau |\grho_{h}^{k} ) - \E_{i,h}(\rho_{i,h}^{k+1} |\grho_{h}^{k}  ) \right) \geqslant 0.
\end{eqnarray}
By standard computations (see for instance \cite{JKO}, \cite{A}) we have 
\begin{eqnarray}
\label{variation distance}
\limsup_{\tau \searrow 0} \frac{1}{\tau}(W_2^2(\rho_\tau,\rho_{i,h}^{k}) -W_2^2(\rho_{i,h}^{k+1},\rho_{i,h}^{k})) \leqslant \int_{\Rn \times \Rn} (y-x) \cdot \xi_i(y) \, d\gamma_{i,h}^k(x,y),
\end{eqnarray}
with $\gamma_{i,h}^k$ is the optimal transport plan in $W_2(\rho_{i,h}^{k},\rho_{i,h}^{k+1})$,
\begin{eqnarray}
\label{variation entropie}
\limsup_{\tau \searrow 0} \frac{1}{\tau}(\F_i(\rho_\tau) -\F_i(\rho_{i,h}^{k+1}) )\leqslant - \int_{\Rn } P_i(\rho_{i,h}^{k+1}(x)) \dive(\xi_i(x))\, dx,
\end{eqnarray}
and
\begin{eqnarray}
\label{variation terme linéaire}
\limsup_{\tau \searrow 0} \frac{1}{\tau}\left(\V_i ( \rho_\tau |\grho_{h}^{k} )   -\V_i ( \rho_{i,h}^{k+1}|\grho_{h}^{k} )   \right) \leqslant  \int_{\Rn }\nabla(V_i [\grho_{h}^{k} ])(x) \cdot \xi_i(x) \rho_{i,h}^{k+1}(x) \,dx.
\end{eqnarray}

If we combine \eqref{premiere varaiation}, \eqref{variation distance}, \eqref{variation entropie} and \eqref{variation terme linéaire}, we get

$$  \int_{\Rn \times \Rn} (y-x) \cdot \xi_i(y) \, d\gamma_{i,h}^k(x,y)+h\int_{\Rn }\nabla(V_i [\grho_{h}^{k} ]) (x)\cdot \xi_i(x)\rho_{i,h}^{k+1}(x)\,dx- h\int_{\Rn } P_i(\rho_{i,h}^{k+1}(x)) \dive(\xi_i(x))\,dx\geqslant 0.$$

And if we replace $\xi_i$ by $-\xi_i$, this inequality becomes an equality. 

To conclude this first part, we choose $\xi_i= \nabla \varphi_i $ and we notice, using Taylor's expansion, that
$$ \varphi_i(x)-\varphi_i(y) = \nabla  \varphi_i (y) \cdot (x-y) +\mathcal{R}[\varphi_i](x,y),$$
with $\mathcal{R}[\varphi_i]$ satisfies
$$ |\mathcal{R}[\varphi_i](x,y)| \leqslant \frac{1}{2} \|D^2 \varphi_i \|_{L^\infty ([0,T) \times \Rn)} |x- y|^2.$$

\item {\it Second step:} let $h>0$, for all $T>0$, let $N$ such that $Nh=T$ ($t_k:=kh$) and for all $(\phi_1, \dots, \phi_l) \in \mathcal{C}^\infty_c ([0,T)\times \Rn)^l$, extend, for all $i \in [\![ 1 ,l ]\!]$, by $\phi_i(0,\cdot)$ on $[-h,0)$, then

\begin{eqnarray*}
\int_0^T \int_{\Rn} \rho_{i,h}(t,x)  \partial_t \phi_i(t,x)\,dxdt &=& \sum_{k=0}^{N} \int_{t_{k-1}}^{t_{k}} \int_{\Rn} \rho_{i,h}^k(x)  \partial_t \phi_i(t,x)\,dxdt\\
&=& \sum_{k=0}^{N} \int_{\Rn} \rho_{i,h}^k(x) (\phi_i(t_{k},x)-\phi_i(t_{k-1},x))\,dx\\
&=& \sum_{k=0}^{N-1} \int_{\Rn} \phi_i(t_{k},x) (\rho_{i,h}^{k}(x) -\rho_{i,h}^{k+1}(x))\,dx - \int_{\Rn} \rho_{i,0}(x) \phi_i(0,x) \, dx.
\end{eqnarray*}
Using the first part with $\varphi_i=\phi_i(t_{k},\cdot)$, we get
 \begin{eqnarray*}
\int_0^T \int_{\Rn} \rho_{i,h}(t,x)  \partial_t \phi_i(t,x) \,dxdt&=&-h\sum_{k=0}^{N-1}\int_{\Rn} P_i(\rho_{i,h}^{k+1}(x))  \Delta \phi_i (t_{k},x)\,dx\\
&+&h\sum_{k=0}^{N-1}\int_{\Rn} \nabla (V_i[\grho_h^{k}] \cdot  \nabla \phi_i(t_{k},x) \rho_{i,h}^{k+1}(x)\,dx\\
&+&\sum_{k=0}^{N-1}\int_{\Rn \times \Rn} \mathcal{R}[\phi_i(t_{k},\cdot)](x,y) d\gamma_{i,h}^k (x,y)\\
&-& \int_{\Rn} \rho_{i,0}(x) \phi_i(0,x) \, dx,
\end{eqnarray*}
for all $i \in [\![ 1 ,l ]\!]$.
\end{itemize}
\end{proof}

The last proposition of this section gives usual estimates in gradient flow theory.

\begin{prop}
\label{estimates}
For all $T <+\infty$ and for all $i \in [\![ 1 ,l ]\!]$, there exists a constant $C <+\infty$ such that for all $k \in \mathbb{N}$ and for all $h$ with $kh\leqslant T$ and let $N=\lfloor \frac{T}{h} \rfloor$, we have
\begin{eqnarray}
& M(\rho_{i,h}^k) \leqslant C, \label{estimation moment}\\
& \F_i(\rho_{i,h}^k) \leqslant C,\label{estimation entropie}\\
& \sum_{k=0}^{N-1} W_2^2(\rho_{i,h}^{k},\rho_{i,h}^{k+1}) \leqslant Ch.\label{estimation distance}
\end{eqnarray}

\end{prop}

\begin{proof}
The proof combines some techniques used in \cite{JKO} et \cite{DFF}. In the following, $i$ is fixed in $[\![ 1 ,l ]\!]$. As $\rho_{i,h}^{k+1}$ is optimal and $\rho_{i,h}^{k}$ is admissible, we have
$$ \E_{i,h}(\rho_{i,h}^{k+1}|\grho_{h}^{k}) \leqslant  \E_{i,h}(\rho_{i,h}^{k}|\grho_{h}^{k}).$$
In other words,
$$\frac{1}{2} W_2^2(\rho_{i,h}^{k},\rho_{i,h}^{k+1}) +h\left(\F_i(\rho_{i,h}^{k+1}) + \V_i ( \rho_{i,h}^{k+1}|\grho_{h}^{k})  \right) \leqslant h\left(\F_i(\rho_{i,h}^{k}) + \V_i ( \rho_{i,h}^{k}|\grho_{h}^{k}) \right).$$

From \eqref{Hyp 2 V}, we know that $V_i[\grho]$ is a $C$-Lipschitz function where $C$ does not depend on the measure. Hence, because of \eqref{Lipschitz W2}, we have

$$ \V_i ( \rho_{i,h}^{k}|\grho_{h}^{k})- \V_i ( \rho_{i,h}^{k+1}|\grho_{h}^{k})\leqslant CW_2(\rho_{i,h}^{k+1},\rho_{i,h}^{k}).$$

Using Young's inequality, we obtain
$$\V_i ( \rho_{i,h}^{k}|\grho_{h}^{k})- \V_i ( \rho_{i,h}^{k+1}|\grho_{h}^{k}) \leqslant C^2h  +\frac{1}{4h}W_2^2(\rho_{i,h}^{k+1},\rho_{i,h}^{k}).$$

It yields
\begin{eqnarray}
\label{maj entropie}
 \frac{1}{4} W_2^2(\rho_{i,h}^{k},\rho_{i,h}^{k+1}) \leqslant h(\F_i(\rho_{i,h}^{k})-\F_i(\rho_{i,h}^{k+1}))+C^2h^2.
\end{eqnarray}
Summing over $k$, we can assert that
\begin{eqnarray*}
\sum_{k=0}^{N-1} \frac{1}{4} W_2^2(\rho_{i,h}^{k},\rho_{i,h}^{k+1}) &\leqslant & h\left(\sum_{k=0}^{N-1} \left(\F_i(\rho_{i,h}^{k})-\F_i(\rho_{i,h}^{k+1}) \right) + C^2T \right)\\
& \leqslant & h \left(\F_i(\rho_{i,0})-\F_i(\rho_{i,h}^{N}) +C^2T \right).
\end{eqnarray*}
But by assumption, $\F_i(\rho_{i,0})<+\infty$ and $-\F_i(\rho)\leqslant C(1+M(\rho))^\alpha$, with $0<\alpha<1$, then 
\begin{eqnarray}
\label{premiere maj somme}
\sum_{k=1}^N \frac{1}{4} W_2^2(\rho_{i,h}^{k},\rho_{i,h}^{k+1}) \leqslant h\left(\F_i(\rho_{i,0})+ C(1+M(\rho_{i,h}^{N}))^\alpha+ C^2T \right).
\end{eqnarray}
Thus we are reduced to prove \eqref{estimation moment}. But
\begin{eqnarray*}
M(\rho_{i,h}^k) &\leqslant & 2 W_2^2( \rho_{i,h}^k,\rho_{i,0}) +2M(\rho_{i,0})\\
&\leqslant & 2 k\sum_{m=0}^{k-1}  W_2^2(\rho_{i,h}^{m},\rho_{i,h}^{m+1})+2M(\rho_{i,0})\\
&\leqslant & 8 kh\left(\F_i(\rho_{i,0})+ C(1+M(\rho_{i,h}^{k}))^\alpha+ C^2T \right) +2M(\rho_{i,0})\\
&\leqslant & 8T\left(\F_i(\rho_{i,0})+ C(1+M(\rho_{i,h}^{k}))^\alpha+ C^2T \right) +2M(\rho_{i,0}).
\end{eqnarray*}
As $\alpha <1$, we get \eqref{estimation moment}. The second line is obtained with the triangle inequality and Cauchy-Schwarz inequality while the third line is obtained because of \eqref{premiere maj somme}. So we have poved \eqref{estimation moment} and \eqref{estimation distance}.

\smallskip
To have \eqref{estimation entropie}, we just have to use \eqref{maj entropie} and to sum. This implies 
$$ \F_i(\rho_{i,h}^k) \leqslant \F_i(\rho_{i,0})+ C^2T,$$
which proves the proposition.

\end{proof}


\section{$\kappa$-flows and gradient estimate.}\label{flow interchange}

Estimates of proposition \ref{estimates} permit to obtain weak convergence in $L^1$ (see proposition \ref{cv faible}). Unfortunately, it is not enough to pass to the limit in the nonlinear diffusion term $P_i(\rho_{i,h})$. In this section, we follow the general strategy developed in \cite{MMCS} and used in \cite{DFM} and \cite{CL} to get an estimate on the gradient of $\rho_{i,h}^{m_i/2}$. This estimate will be used in proposition \ref{cv forte} to have a strong convergence of $\rho_{i,h}$ in $L^{m_i}(]0,T[\times \Rn)$. In the following, we are only interested by the case where $m_i >1$ because if $m_i=1$, $P_i(\rho_{i,h})=\rho_{i,h}$ and the weak convergence is enough to pass to the limit in proposition \ref{equation discrete}. In the first part of this section, we recall the definition of $\kappa$-flows (or contractive gradient flow) and some results on the dissipation of $\F_i +\V_i$ then, in the second part, we use these results with the heat flow to find an estimate on the gradient.


\subsection{$\kappa$-flows.}

\begin{de}
A semigroup $\mathfrak{S}_\Psi \, : \, \R^+ \times \mathcal{P}_{2}^{ac}(\Rn) \rightarrow \mathcal{P}_{2}^{ac}(\Rn)$ is a $\kappa$-flow for the functional $\Psi \, :\,\mathcal{P}_{2}^{ac}(\Rn)\rightarrow \R \cup \{ +\infty\}$ with respect to  $W_2$ if, for all $\rho \in \mathcal{P}_{2}^{ac}(\Rn)$, the curve $s \mapsto \mathfrak{S} _\Psi^s[\rho]$ is absolutely continuous on $\R^+$ and satisfies the evolution variational inequality (EVI) 
\begin{eqnarray}
\label{EVI}
\frac{1}{2}\frac{d^+}{d\sigma}\mid_{\sigma=s} W_2^2(\mathfrak{S} _\Psi^s[\rho], \tilde{\rho}) +\frac{\kappa}{2}W_2^2(\mathfrak{S} _\Psi^s[\rho], \tilde{\rho}) \leqslant \Psi(\tilde{\rho})-\Psi(\mathfrak{S} _\Psi^s[\rho]),
\end{eqnarray}
for all $s>0$ and for all $\tilde{\rho} \in \mathcal{P}_{2}^{ac}(\Rn)$ such that $\Psi(\tilde{\rho}) < +\infty$, where 
$$ \frac{d^+}{dt}f(t) := \limsup_{s\rightarrow 0^+} \frac{f(t+s) -f(t)}{s}.$$
\end{de}

In \cite{AGS}, the authors showed that the fact a functional admits a $\kappa$-flow is equivalent to $\lambda$-displacement convexity (see section \ref{uniqueness} for definition).

The next two lemmas give results on the variations of $\rho_{i,h}^k$ along specific $\kappa$-flows and are extracted from \cite{DFM}. The goal is to use them with the heat flow.\\

\begin{lem}
\label{lem dissipation}
Let $\Psi \, :\,\mathcal{P}_{2}^{ac}(\Rn)\rightarrow \R \cup \{ +\infty\}$ l.s.c on $\mathcal{P}_{2}^{ac}(\Rn)$ which possesses a $\kappa$-flot $\mathfrak{S}_\Psi$. Define the dissipation $\D_{i,\Psi}$ along $\mathfrak{S}_\Psi$ by
$$ \D_{i,\Psi}(\rho | \boldsymbol{\mu}) := \limsup_{s \searrow 0} \frac{1}{s} \left(\F_i(\rho)-\F_i(\mathfrak{S}_\Psi^s[\rho]) + \V_i(\rho | \boldsymbol{\mu})-\V_i(\mathfrak{S} _\Psi^s[\rho]|\boldsymbol{\mu})\right)$$
for all $\rho \in \mathcal{P}_{2}^{ac}(\Rn)$ and $\boldsymbol{\mu} \in \mathcal{P}_{2}^{ac}(\Rn)^l$.\\
If $\rho_{i,h}^{k-1}$ et $\rho_{i,h}^{k}$ are two consecutive steps of the semi-implicit JKO scheme, then
\begin{eqnarray}
\label{dissipation}
\Psi(\rho_{i,h}^{k-1})-\Psi(\rho_{i,h}^{k}) \geqslant h\D_{i,\Psi}(\rho_{i,h}^{k} | \grho_{h}^{k-1}) + \frac{\kappa}{2}W_2^2(\rho_{i,h}^{k},\rho_{i,h}^{k-1}).
\end{eqnarray}

\end{lem}

\begin{proof}
Since the result is trivial if $\Psi(\rho_{i,h}^{k-1}) =+\infty$, we assume $\Psi(\rho_{i,h}^{k-1}) <+\infty$.\\
Thus we can use the EVI inequality \eqref{EVI} with $\rho:=\rho_{i,h}^k$ and $\tilde{\rho}:=\rho_{i,h}^{k-1}$. We obtain  
$$ \Psi(\rho_{i,h}^{k-1})-\Psi(\s_\Psi^s[\rho_{i,h}^{k}]) \geqslant \frac{1}{2}\frac{d^+}{d\sigma}\mid_{\sigma=s} W_2^2(\mathfrak{S} _\Psi^s[\rho_{i,h}^{k}], \rho_{i,h}^{k-1}) +\frac{\kappa}{2}W_2^2(\mathfrak{S} _\Psi^s[\rho_{i,h}^{k}], \rho_{i,h}^{k-1}).$$

By lower semi-continuity of $\Psi$, we have
\begin{eqnarray*}
\Psi(\rho_{i,h}^{k-1})-\Psi(\rho_{i,h}^{k}) &\geqslant & \Psi(\rho_{i,h}^{k-1})-\liminf_{s \searrow 0} \Psi(\s_\Psi^s[\rho_{i,h}^{k}])\\
& \geqslant & \limsup_{s \searrow 0} \left( \Psi(\rho_{i,h}^{k-1})-\Psi(\s_\Psi^s(\rho_{i,h}^{k})) \right)\\
& \geqslant & \limsup_{s \searrow 0} \left(\frac{1}{2}\frac{d^+}{d\sigma}\mid_{\sigma=s} W_2^2(\mathfrak{S} _\Psi^s[\rho_{i,h}^{k}], \rho_{i,h}^{k-1}) \right) +\frac{\kappa}{2}W_2^2(\rho_{i,h}^{k}, \rho_{i,h}^{k-1}).
\end{eqnarray*}
The last line is obtained thanks to the $W_2$-continuity of $s \mapsto \s_\Psi^s[\rho_{i,h}^k]$ in $s=0$. Moreover, the absolute continuity of $s \mapsto \mathfrak{S} _\Psi^s[\rho_{i,h}^k]$ implies
\begin{eqnarray*}
\limsup_{s \searrow 0} \left(\frac{1}{2}\frac{d^+}{d\sigma}\mid_{\sigma=s} W_2^2(\mathfrak{S} _\Psi^s[\rho_{i,h}^{k}], \rho_{i,h}^{k-1}) \right) \geqslant \limsup_{s \searrow 0} \frac{1}{s}\left( W_2^2(\mathfrak{S} _\Psi^s[\rho_{i,h}^{k}], \rho_{i,h}^{k-1}) - W_2^2(\rho_{i,h}^{k}, \rho_{i,h}^{k-1}) \right).
\end{eqnarray*}
But since $\rho_{i,h}^{k}$ minimizes $\E_{i,h}(\cdot | \grho_h^{k-1})$, we get, for all $s \geqslant 0$,
\begin{eqnarray*}
 W_2^2(\mathfrak{S} _\Psi^s[\rho_{i,h}^{k}], \rho_{i,h}^{k-1}) - W_2^2(\rho_{i,h}^{k}, \rho_{i,h}^{k-1})  \geqslant  2h \left( \F_i(\rho_{i,h}^{k}) -\F_i(\mathfrak{S} _\Psi^s[\rho_{i,h}^{k}]) \right) + 2h \left( \V_i(\rho_{i,h}^{k} | \grho_h^{k-1})-\V_i(\mathfrak{S} _\Psi^s[\rho_{i,h}^{k}]|\grho_h^{k-1}) \right).
\end{eqnarray*}

This concludes the proof.

\end{proof}

\begin{cor}
\label{cor dissipation}
Under the same hypotheses as in lemma \ref{lem dissipation}, let $\s_\Psi$ a $\kappa$-flow such that, for all $k\in \mathbb{N}$, the curve $s \mapsto \s_\Psi^s[ \rho^k_{i,h}]$ lies in $L^{m_i}(\Rn)$, is differentiable for $s>0$ and is continuous at $s=0$.\\
Let $\K_{i,\Psi} \, : \, \mathcal{P}_{2}^{ac}(\Rn) \mapsto ]-\infty, +\infty]$ be a functional such that
\begin{eqnarray}
\label{hyp K}
\liminf_{s \searrow 0} \left( -\frac{d}{d\sigma}_{|_{\sigma=s}} \left(\F_i(\s_\Psi^\sigma [\rho^k_{i,h}]) + \V_i(\mathfrak{S} _\Psi^\sigma[\rho_{i,h}^{k}]|\grho_{h}^{k-1}) \right) \right) \geqslant \K_{i,\Psi}(\rho_{i,h}^{k}|\grho_h^{k-1}).
\end{eqnarray}
Then, for all $k \in \mathbb{N}$,
\begin{eqnarray}
\Psi(\rho_{i,h}^{k-1})-\Psi(\rho_{i,h}^{k}) \geqslant h \K_{i,\Psi}(\rho_{i,h}^{k}|\grho_h^{k-1}) + \frac{\kappa}{2}W_2^2(\rho_{i,h}^{k}, \rho_{i,h}^{k-1}).
\end{eqnarray}

\end{cor}

\begin{proof}
It is sufficient to show that $\D_{i,\Psi}(\cdot |\grho_h^{k-1})$ is bounded below by $\K_{i,\Psi}(\cdot |\grho_h^{k-1})$. The proof is as in corollary 4.3 of \cite{DFM}. The hypothese of $L^{m_i}$-régularity on $\s_\Psi$ imply that $s \mapsto \F_i(\s_\Psi^s[\rho^k_{i,h}])$ is differentiable for $s>0$ and continuous at $s=0$. We have the same regularity for $s \mapsto \V_i(\s_\Psi^s[\rho^k_{i,h}]|\grho_h^{k-1})$. By the fundamental theorem of calculus,
\begin{eqnarray*}
\D_{i,\Psi}(\rho_{i,h}^{k}|\grho_h^{k-1}) &=& \limsup_{s \searrow 0} \frac{1}{s} \left(\F_i(\rho_{i,h}^{k})-\F_i(\mathfrak{S}_\psi^s[\rho_{i,h}^{k}]) + \V_i(\rho_{i,h}^{k} | \grho_h^{k-1})-\V_i(\mathfrak{S} _\Psi^s[\rho_{i,h}^{k}]|\grho_h^{k-1})\right)\\
&=&\limsup_{s \searrow 0} \int_0^1\left( -\frac{d}{d\sigma}_{|_{\sigma=sz}} \left(\F_i(\s_\Psi^\sigma [\rho^k_{i,h}]) + \V_i(\mathfrak{S} _\Psi^\sigma[\rho_{i,h}^{k}]|\grho_h^{k-1}) \right) \right) \,dz\\
&\geqslant & \int_0^1 \liminf_{s \searrow 0}\left( -\frac{d}{d\sigma}_{|_{\sigma=sz}} \left(\F_i(\s_\Psi^\sigma [\rho^k_{i,h}]) + \V_i(\mathfrak{S} _\Psi^\sigma[\rho_{i,h}^{k}]|\grho_h^{k-1}) \right) \right) \,dz \geqslant \K_{i,\Psi}(\rho_{i,h}^{k}|\grho_h^{k-1}).
\end{eqnarray*}
The last line is obtained by Fatou's lemma and assumption \eqref{hyp K}. To conclude we apply lemma \ref{lem dissipation}.
\end{proof}

\subsection{Gradient estimate.}

\begin{prop}
\label{estimation gradient}
For all $i \in [\![ 1 ,l ]\!]$ such that $m_i>1$, there exists a constant $C$ which depends only on  $\rho_{i,0}$ such that 
$$ \|\rho_{i,h}^{m_i/2}\|_{L^2([0,T];H^1(\Rn))} \leqslant C(1+T)$$
for all $T>0$. 
\end{prop}

Before starting the proof of the proposition \ref{estimation gradient}, we recall the definition of the Entropy functional ,
$$ E(\rho) = \int_{\Rn} \rho \log\rho, \text{ for all } \rho \in \mathcal{P}^{ac}(\Rn).$$
We know that this functional possesses a $\kappa$-flow, with $\kappa=0$ which is given by the heat semigroup (see for instance \cite{DS}, \cite{JKO} or \cite{V1}). In other words, for a given $\eta_0 \in \mathcal{P}_{2}^{ac}(\Rn)$, the curve $s \mapsto \eta(s):= \s_E^s[\eta_0]$ solves

$$ \left\{ 
\begin{array}{ll}
\partial_s \eta = \Delta \eta & \text{   on } \mathbb{R}^+ \times \Rn,\\

\eta(0)=\eta_0 &\text{   on } \Rn,\\
\end{array}\right.$$
in the classical sense. $\eta(s)$ is a positive density for all $s>0$ and is continuously differentiable as a map from $\R^+$ to $\mathcal{C}^\infty \cap L^1(\Rn)$. Moreover, if $\eta_0 \in L^m(\Rn)$, then $\eta(s)$ converges to $\eta_0$ in $ L^m(\Rn)$ when $s\searrow 0$.

\begin{proof}
Based on the facts set out above, $\s_E$ satisfies the hypotheses of the corollary \ref{cor dissipation}. We just have to define a suitable lower bound $\K_{i,E}$ to use it. The spatial regularity of $\eta(s)$ for all $s >0$ allows the following calculations. Thus for all $\boldsymbol{\mu} \in \mathcal{P}_{2}^{ac}(\Rn)^l$, we have
\begin{eqnarray*}
\partial_s \left( \F_i(\s_E^s[\eta_0]) + \V_i(\s_E^s[\eta_0] | \boldsymbol{\mu}) \right) & = & \int_{\Rn} \partial_s F_i(\eta) \,dx + \int_{\Rn} V_i[\boldsymbol{\mu}]\partial_s \eta(s,x) \, dx\\
& = & \int_{\Rn} F_i'(\eta(s,x)) \Delta \eta(s,x) \,dx + \int_{\Rn} V_i[\boldsymbol{\mu}]\Delta \eta(s,x) \, dx\\
& = & -\int_{\Rn} F_i''(\eta(s,x)) |\nabla \eta(s,x)|^2 \,dx + \int_{\Rn} \Delta( V_i[\boldsymbol{\mu}])\eta(s,x) \, dx.
\end{eqnarray*}
According to \eqref{hyp F}, $F_i''(x) \geqslant Cx^{m_i-2}$ thus
\begin{eqnarray*}
 \partial_s \left( \F_i(\s_E^s[\eta_0]) + \V_i(\s_E^s[\eta_0] | \boldsymbol{\mu}) \right)& \leqslant & -C\int_{\Rn} \eta(s,x)^{m_i-2} |\nabla \eta(s,x)|^2 \,dx + \int_{\Rn} \Delta( V_i[\boldsymbol{\mu}])\eta(s,x) \, dx\\
 &\leqslant & -C\int_{\Rn}|\nabla \eta(s,x)^{m_i/2}|^2 \,dx + \int_{\Rn} \Delta( V_i[\boldsymbol{\mu}])\eta(s,x) \, dx.
\end{eqnarray*}

Then we define
$$ \K_{i,E}(\rho | \boldsymbol{\mu}) := C\int_{\Rn}  |\nabla (\rho(s,x)^{m_i/2})|^2 \,dx - \int_{\Rn} \Delta (V[\boldsymbol{\mu}])\rho(s,x) \, dx.$$

\bigskip

We shall now establish that $\K_{i,E}$ satisfies \eqref{hyp K}. First of all, we notice that
\begin{eqnarray}
\label{liminf somme}
\liminf_{s \searrow 0} \left( -\frac{d}{d\sigma}_{|_{\sigma=s}} (\F_i(\s_E^\sigma [\rho^k_{i,h}]) + \V_i(\mathfrak{S} _E^\sigma[\rho_{i,h}^{k}]|\grho_h^{k-1}) \right) &\geqslant &\liminf_{s \searrow 0} \left( -\frac{d}{d\sigma}_{|_{\sigma=s}} \F_i(\s_E^\sigma [\rho^k_{i,h}]) \right)\\
& + &\liminf_{s \searrow 0} \left( -\frac{d}{d\sigma}_{|_{\sigma=s}} \V_i(\mathfrak{S} _E^\sigma[\rho_{i,h}^{k}]|\grho_h^{k-1}) \right).
\end{eqnarray}

Thanks to the proof of lemma 4.4 and with lemma A.1 of \cite{DFM}, we obtain
\begin{eqnarray}
\label{liminf energie}
\liminf_{s \searrow 0} \left( -\frac{d}{d\sigma}_{|_{\sigma=s}} (\F_i(\s_E^\sigma [\rho^k_{i,h}]) \right) \geqslant C\int_{\Rn}  |\nabla (\rho_{i,h}^k(x)^{m_i/2})|^2 \,dx.
\end{eqnarray}

Moreover, as $\s_E^s$ is continuous in $L^1(\Rn)$ at $s=0$ and according to \eqref{Hyp 2 V}, 
\begin{eqnarray}
\label{liminf V}
\liminf_{s \searrow 0} \left( -\frac{d}{d\sigma}_{|_{\sigma=s}} \V_i(\mathfrak{S} _E^\sigma[\rho_{i,h}^{k}]|\grho_h^{k-1}) \right) \geqslant - \int_{\Rn} \Delta (V_i[\grho_h^{k-1}])\rho_{i,h}^k(x) \, dx.
\end{eqnarray}

The combination of \eqref{liminf somme}, \eqref{liminf energie} and \eqref{liminf V} gives \eqref{hyp K} for $\K_{i,E}$. We apply corollary \ref{cor dissipation} and we get

\begin{eqnarray}
E(\rho_{i,h}^{k-1})-E(\rho_{i,h}^{k}) \geqslant h \K_{i,E}(\rho_{i,h}^{k}|\grho_h^{k-1}).
\end{eqnarray}

but since $\Delta( V_i[\grho] )\in L^\infty(\Rn)$ uniformly on $\grho$ \eqref{Hyp 2 V}, 
$$ Ch\int_{\Rn}  |\nabla (\rho_{i,h}^k(x)^{m_i/2})|^2 \,dx \leqslant E(\rho_{i,h}^{k-1})-E(\rho_{i,h}^{k}) +Ch.$$

Now we sum on $k$ from $1$ to $N=\lfloor \frac{T}{h}\rfloor$
\begin{eqnarray}
Ch\sum_{k=1}^N \| \nabla (\rho_{i,h}^k(x)^{m_i/2}) \|_{L^2(\Rn)}^2 \leqslant E(\rho_{i,0}) -E(\rho^N_{i,h}) +CT.
\end{eqnarray}

According to \cite{JKO} and \cite{DFM}, there exists a constant $C>0$ and $0<\alpha<1$ such that for all $\rho \in \mathcal{P}_{2}^{ac}(\Rn)$,
$$ -C(1+ M(\rho))^\alpha \leqslant E(\rho) \leqslant C \F_i(\rho).$$

Since for all $k,h$, $M(\rho_{i,h}^k)$ is bounded, according to \eqref{estimation moment} and the fact that $\F_i(\rho_{i,0}) < +\infty$ by \eqref{hyp energie initiale}, we have
$$ h\sum_{k=1}^N \| \nabla (\rho_{i,h}^k(x)^{m_i/2})\|_{L^2(\Rn)}^2  \leqslant C(1+T).$$

To conclude the proof, we use \eqref{hyp F} and \eqref{estimation entropie}.

\end{proof}

\section{Passage to the limit.}\label{limit}

\subsection{Weak and strong convergences.}

The first convergence result is obtained using the estimates on the distance \eqref{estimation distance} and on the energy $\F_i$ \eqref{estimation entropie}.

\begin{prop}
\label{cv faible}
Every sequences $(h_k)_{k\in \mathbb{N}}$ of time steps which tends to $0$ contains a subsequence, non-relabelled, such that
$\rho_{i,h_k}$ converges, uniformly on compact time intervals, in $W_2$ to a $\frac{1}{2}$-Hölder function $\rho_i \, : \, [0,+\infty[ \rightarrow \mathcal{P}_{2}^{ac}(\Rn)$. 
\end{prop}

\begin{proof}
The estimation on the sum of distances gives us for all $t,s\leqslant 0$,
$$ W_2(\rho_{i,h}(t,\cdot),\rho_{i,h}(s,\cdot)) \leqslant C(|t-s|+h)^{1/2},$$
with $C$ independ of $h$.

According to the proposition 3.3.1 of \cite{AGS} and using a diagonal argument, at least for a subsequence, for all $i$, $\rho_{i,h_k}$ converges uniformly on compact time intervals in $W_2$ to a $\frac{1}{2}$-Hölder function $\rho_i \, : \, [0,+\infty[ \rightarrow \mathcal{P}_{2}(\Rn)$. To conclude we show that for all $t\geqslant 0$, $\rho(t,\cdot) \in \mathcal{P}_{2}^{ac}(\Rn)$. But as $F_i$ is superlinear, Dunford-Pettis' theorem completes the proof.

\end{proof}

With the previous proposition, we can pass to the limit in the case $m_i=1$ because $P_i(\rho_{i,h})=\rho_{i,h}$ and in the term $\nabla(V_i[\rho_{i,h}])$ thanks to the hypothesis \eqref{Hyp 3 V}. Unfortunately, it is not enough to pass to the limit in $P_i(\rho_{i,h})$ when $m_i>1$. In the next proposition, we use proposition \ref{estimation gradient} to get a stronger convergence.

\begin{prop}
\label{cv forte}
For all $i\in [\![ 1,l ]\!] $ such that $m_i>1$, $\rho_{i,h}$ converges to $\rho_i$ in $L^{m_i}( ]0,T[ \times \Rn)$ and $P_i(\rho_{i,h}) $ converges to $P_i(\rho_i)$ in $L^1(]0,T[ \times \Rn)$, for all $T>0$. 
\end{prop}

The proof of this proposition is obtained by using an extention of Aubin-Lions lemma given by Rossi and Savaré in \cite{RS} (theorem 2) and recalled in \cite{DFM} (theorem 4.9).

\begin{thm}[th. 2 in \cite{RS}]
\label{Rossi-Savaré}
On a Banach space $X$, let be given
\begin{itemize}
\item a normal coercive integrand $\G \, : \, X \rightarrow \R^+$, i.e, $\G$ is l.s.c and its sublevels are relatively compact in $X$,
\item a pseudo-distance $g \, : \, X\times X \rightarrow [0,+\infty]$, i.e, $g$ is l.s.c and $ \left[ g(\rho,\mu)=0, \, \rho,\mu \in X \text{ with } \G(\rho),\G(\mu) < \infty \right] \Rightarrow \rho=\mu$.
\end{itemize}

Let $U$ be a set of measurable functions $u \, : \, ]0,T[ \rightarrow X$ with a fixed $T>0$. Under the hypotheses that 
\begin{align}
\label{condition U}
\sup_{u\in U} \int_0^T \G(u(t)) \,dt <+\infty \qquad \text{  and  } \qquad \lim_{h\searrow 0} \sup_{u \in U} \int_0^{T-h} g(u(t+h),u(t)) \, dt =0,
\end{align} 
$U$ contains a subsequence $(u_n)_{n \in \N}$ which converges in measure with respect to $t\in]0,T[$ to a limit $u_\star \, : \, ]0,T[ \rightarrow X$.
\end{thm}

To apply this theorem, we define on $X:=L^{m_i}( \Rn)$, as in \cite{DFM}, $g$ by
$$ g(\rho,\mu) :=\left\{ \begin{array}{ll}
W_2(\rho,\mu) & \text{if } \rho,\mu \in \mathcal{P}_2(\Rn),\\
+\infty & \text{otherwise},
\end{array}\right.$$
and $\G_i$ by
$$ \G_i(\rho):= \left\{ \begin{array}{ll}
\|\rho^{m_i/2} \|_{H^1(\Rn)} + M(\rho) & \text{if } \rho \in \mathcal{P}_{2}^{ac}(\Rn) \text{ and } \rho^{m_i/2} \in H^1(\Rn),\\
+\infty & \text{otherwise}.
\end{array}\right.$$

Now, we show that $\G_i$ satisfies theorem \ref{Rossi-Savaré} conditions. 
\begin{lem}
\label{sublevel}
For all $i\in [\![ 1,l ]\!] $ such that $m_i>1$, $\G_i$ is l.s.c and its sublevels are relatively compact in $L^{m_i}(\Rn)$.
\end{lem}

\begin{proof}
The l.s.c of $\G_i$ on $L^{m_i}(\Rn)$ follows from lemma A.1 in \cite{DFM}. To complete the proof we have to show that sublevels $A_c:=\left\{\rho \in L^{m_i}(\Rn) \, |\, \G_i(\rho) \leqslant c \right\}$ of $\G_i$ are relatively compact in $L^{m_i}(\Rn)$.

To do this, we prove that $B_c := \left\{ \eta=\rho^{m_i/2} \, | \, \rho \in A_c \right\}$ is relatively compact in $L^2(\Rn)$ and since the map $j \, : \, L^2(\Rn) \rightarrow L^{m_i}(\Rn)$, with $j(\eta)=\eta^{2/m_i}$, is continuous, $A_c=j(B_c)$ will be relatively compact in $L^{m_i}(\Rn)$.\\ 
We want to apply the Frechét-Kolmogorov theorem to show that $B_c$ is relatively compact in $L^2(\Rn)$.

\begin{itemize}
\item {\it $B_c$ is bounded in $L^2(\Rn)$:} Since $\eta^2 =\rho^{m_i}$ with $\G_i(\rho) \leqslant c$, it is straightforward to see
$$ \int_{\Rn} \eta^2 \leqslant c.$$

\item {\it $B_c$ is tight under translations:} for every $\eta \in B_c$ and $h\in\Rn$ we have that
$$ \int_{\Rn} |\eta(x+h)-\eta(x)|^2 \,dx \leqslant |h|^2 \int_{\Rn} \left| \int_0^1 |\nabla \eta(x+zh) |\,dz \right|^2 \,dx \leqslant  |h|^2\int_{\Rn} |\nabla \eta(x) |^2 \,dx \leqslant c|h|^2,$$
thus the left hand side converges to $0$ uniformly on $B_c$ as $|h|\searrow 0$.

\item {\it Elements of $B_c$ are unifomly decaying at infinity:} For all $\eta \in B_c$ and $R >0$, we have 
$$ \int_{|x|>R} \eta^2 \, dx \leqslant \frac{1}{R^{1/n}} \int_{\Rn} |x|^{1/n} \eta^{1/nm_i} \eta^{2-1/nm_i} \,dx.$$

If we use Hölder inequality with $p=2n$ and $q=\frac{2n}{2n-1}$, we get

$$ \int_{|x|>R} \eta^2 \, dx \leqslant \frac{1}{R^{1/n}} \left( \int_{\Rn} |x|^{2}\eta^{2/m_i} \right)^{1/2n}\left( \int_{\Rn}\eta^{2(2m_i-1/n)/m_i(2-1/n)} \right)^{\frac{2n-1}{2n}}.$$

As $\eta^{2/m_i}=\rho$ with  $\G_i(\rho) \leqslant c$, we have
$$\int_{\Rn} |x|^{2}\eta^{2/m_i} \leqslant c.$$

To bound the other term we use the Gagliardo-Nirenberg inequality: for $1\leqslant q,r \leqslant +\infty$, we have
$$ \| u \|_{L^p} \leqslant C \| \nabla u \|_{L^r}^\alpha \|u\|_{L^q}^{1-\alpha},$$
for all $0 < \alpha <1$ and for $p$ given by
$$ \frac{1}{p}= \alpha \left(\frac{1}{r}-\frac{1}{n} \right)+(1-\alpha)\frac{1}{q}.$$

We choose $p=\frac{2(2m_i-1/n)}{m_i(2-1/n)}$, $q=r=2$ and $ \alpha=\frac{m_i-1}{2(2m_i-1/n)}$ (since $m_i>1$ we have $0 < \alpha <1$) then we obtain:


$$ \int_{\Rn}\eta^{2(2m_i-1/n)/m_i(2-1/n)} \leqslant \left(\int_{\Rn} |\nabla \eta |^2 \right)^{\alpha p/2}\left(\int_{\Rn} \eta^2 \right)^{(1-\alpha)p/2} .$$
but since $\eta =\rho^{m_i/2}$ with $\G_i(\rho) \leqslant c$, the second term is bounded then
$$ \int_{|x|>R} \eta^2 \, dx \leqslant \frac{C}{R^{1/n}} \rightarrow 0,$$
as $R$ goes to $+\infty$.
\end{itemize}

We conclude thanks to Frechét-Kolmogorov theorem.
\end{proof}

\begin{proof}[Proof of the proposition \ref{cv forte}] We want to apply theorem \ref{Rossi-Savaré} with $X:=L^{m_i}(\Rn)$, $\G:=\G_i$, $g$ and $U:=\{ \rho_{i,h_k} \, | \, k \in \N\}$. According to lemma \ref{sublevel}, $\G_i$ satisfies the hypotheses of the theorem. It's obvious that it is the same for $g$. Thus we only have to check conditions for $U$. The first condition is satisfied because of \eqref{estimation moment} and \eqref{estimation gradient} and the second is satisfied because of \eqref{estimation distance} (the proof is done in \cite{DFM} proposition 4.8, for example).

According to theorem \ref{Rossi-Savaré} and using a diagonal argument, there exists a subsequence, not-relabeled, such that for all $i$ with $m_i>1$, there exists $\tilde{\rho}_i \, : \, ]0,T[ \rightarrow L^{m_i}(\Rn)$ such that $\rho_{i,h_k}$ converges in measure with respect to $t$ in $L^{m_i}(\Rn)$ to $\tilde{\rho}_i$. Moreover, as $\rho_{i,h_k}(t)$ converges in $W_2$ for all $t\in [0,T]$ to $\rho_i(t)$ (proposition \ref{cv faible}) then $\tilde{\rho}_i=\rho_i$. Now since convergence in measure implies a.e convergence up to a subsequence, we may also assume that $\rho_{i,h_k}(t)$ converges strongly in $L^{m_i}(\Rn)$ to $\rho_i(t)$ $t$-a.e. Now, thanks to \eqref{estimation entropie} and \eqref{hyp F} we have
$$ \int_{\Rn} \rho_{i,h}^{m_i} (t,x) \,dx \leqslant C \F_i(\rho_{i,h}(t,\cdot)) \leqslant C,$$
then Lebesgue's dominated convergence theorem implies that $\rho_{i,h_k}$ converges strongly in $L^{m_i}(]0,T[\times\Rn)$ to $\rho_i$.

To conclude the proof we have to show that $P_i(\rho_{i,h}) $ converges to $P_i(\rho_i)$ in $L^1(]0,T[ \times \Rn)$.
First of all, up to a subsequence, we may assume that there exists $g \in L^{m_i}(]0,T[ \times \Rn)$ such that
$$ \rho_{i,h_k} \rightarrow \rho_i \; (t,x)\text{-a.e   and } \rho_{i,h_k} \leqslant g \; (t,x)\text{-a.e}.$$
Thus according to \eqref{hyp F}
$$ P_i(\rho_{i,h_k}) \rightarrow P_i(\rho_i) \; (t,x)\text{-a.e   and } 0 \leqslant P(\rho_{i,h_k}) \leqslant C(\rho_{i,h_k}+g^{m_i})\; (t,x)\text{-a.e}.$$
So when we pass to the limit we have $(t,x)$-a.e
$$  0 \leqslant P(\rho_{i}) \leqslant C(\rho_{i}+g^{m_i})\in L^1(]0,T[ \times \Rn).$$

Then $C(\rho_{i,h_k}+\rho_{i}+2g^{m_i}) -|P_i(\rho_{i,h_k}) - P_i(\rho_i)| \geqslant 0$ and
\begin{eqnarray*}
2CT+2C\iint_{]0,T[ \times \Rn} g(x)^{m_i} \, dxdt &=& \iint_{]0,T[ \times \Rn} \liminf \left(C(\rho_{i,h_k}+\rho_{i}+2g^{m_i}) -|P_i(\rho_{i,h_k}) - P_i(\rho_i)| \right)\\
&\leqslant & 2CT+2C\iint_{]0,T[ \times \Rn} g^{m_i}(x) \, dxdt +\liminf \iint_{]0,T[ \times \Rn}(-|P_i(\rho_{i,h_k}) - P_i(\rho_i)| )\\
&\leqslant & 2CT+2C\iint_{]0,T[ \times \Rn} g^{m_i}(x) \, dxdt -\limsup \iint_{]0,T[ \times \Rn}|P_i(\rho_{i,h_k}) - P_i(\rho_i)|. 
\end{eqnarray*}
To do these computations, we used that $\|\rho_{i,h_k}\|_{L^1(]0,T[ \times \Rn)}=\|\rho_{i}\|_{L^1(]0,T[ \times \Rn)}=T$ and Fatou's lemma. Since $g \in L^{m_i}(]0,T[ \times \Rn)$, we obtain
$$  \limsup \iint_{]0,T[ \times \Rn}|P_i(\rho_{i,h_k}) - P_i(\rho_i)| \leqslant 0,$$
which concludes the proof.

\end{proof}

\subsection{Limit of the discrete system.}

In this section, we pass to the limit in the discrete system of propsosition \ref{equation discrete}. In the following, we consider $\phi_i \in \mathcal{C}^\infty_c([0,T)\times\Rn)$ and $N=\lfloor \frac{T}{h} \rfloor$.

\begin{proof}[ proof of theorem \ref{existence Rn}.] We will pass to the limit in all terms in proposition \ref{equation discrete}.

\smallskip
\begin{itemize}
\item {\it Convergence of the remainder term:} By definition of $\mathcal{R}$, we have
$$ \int_{\Rn\times \Rn} \mathcal{R}[\phi_i(t_{k},\cdot)](x,y) d\gamma_{i,h}^k (x,y) \leqslant \frac{1}{2} \|\nabla ^2 \phi_i \|_{L^\infty ([0,T]\times\Rn)} W_2^2(\rho_{i,h}^{k},\rho_{i,h}^{k+1}).$$
and according to the estimate \eqref{estimation distance}, we get
$$ \left| \sum_{k=0}^{N-1}\int_{\Rn\times \Rn} \mathcal{R}[\phi_i(t_{k},\cdot)](x,y) d\gamma_{i,h}^k (x,y) \right| \leqslant C\sum_{k=0}^{N-1}W_2^2(\rho_{i,h}^{k},\rho_{i,h}^{k+1}) \leqslant Ch \rightarrow 0.$$

\item {\it Convergence of the linear term:}
\begin{eqnarray*}
\left|\int_0^T \int_{\Rn} \rho_{i,h}(t,x) \partial_t \phi_i(t,x) \,dxdt-\int_0^T \int_{\Rn} \rho_i(t,x) \partial_t \phi_i(t,x)\,dxdt \right|\leqslant CT \sup_{t\in [0,T]} W_2(\rho_{i,h}(t,\cdot),\rho_i(t,\cdot))\rightarrow 0,
\end{eqnarray*}
when $h \searrow 0$ because of propostion \ref{cv faible}.

\item {\it Convergence of the diffusion term:} 
\begin{eqnarray*}
&&\left|h\sum_{k=0}^{N-1} \int_{\Rn} P_i(\rho_{i,h}^{k+1}(x)) \cdot \Delta \phi_i(t_k,x)  \, dx- \int_0^T \int_{\Rn} P_i(\rho_i(t,x)) \Delta \phi_i(t,x)  \,dxdt \right|\\
&\leqslant & CT\| D^3 \phi_i \|_{L^\infty}h +\left|   \int_0^T \int_{\Rn}\left(P_i(\rho_{i,h}(t,x))-P_i(\rho(t,x)) \right)\Delta \phi_i(t,x)  \,dxdt \right|.
\end{eqnarray*}
If $m_i=1$, the right hand side converges to $0$ because of proposition \ref{cv faible} and otherwise it goes to $0$ because of proposition \ref{cv forte}. 

\item {\it Convergence of the interaction term:}
\begin{eqnarray*}
&&\left| h\sum_{k=0}^{N-1}\int_{\Rn}  \nabla(V_i [\grho_{h}^{k}])(x)\cdot \nabla \phi_i(t_k,x)\rho_{i,h}^{k+1}(x) \,dx -\int_0^T \int_{\Rn} \nabla(V_i[\grho(t,\cdot)])(x)\cdot \nabla \phi_i(t,x)\rho_i(t,x) \,dxdt \right|\\
&\leqslant & \left| h\sum_{k=0}^{N-1}\int_{\Rn}  \nabla(V_i [\grho_{h}^{k}])(x)\cdot \nabla \phi_i(t_k,x)\rho_{i,h}^{k+1}(x) \,dx -\sum_{k=0}^{N-1}\int_{t_k}^{t_{k+1}} \int_{\Rn} \nabla(V_i[\grho(t,\cdot)])(x)\cdot \nabla \phi_i(t_k,x)\rho_{i,h}^{k+1}(x) \,dxdt  \right|\\
&+& \left|\sum_{k=0}^{N-1}\int_{t_k}^{t_{k+1}} \int_{\Rn} \nabla(V_i[\grho(t,\cdot)])(x)\cdot (\nabla \phi_i(t_k,x) - \nabla \phi_i(t,x) )\rho_{i,h}^{k+1}(x) \,dxdt \right|\\
&+& \left|\sum_{k=0}^{N-1}\int_{t_k}^{t_{k+1}} \int_{\Rn} \nabla(V_i[\grho(t,\cdot)])(x)\cdot \nabla \phi_i(t,x) (\rho_{i,h}^{k+1}(x) -\rho_i(t,x)) \,dxdt \right|\\
&\leqslant & J_1+J_2+J_3.
\end{eqnarray*}

\begin{itemize}
\item As $\rho_{i,h}$ converges weakly $L^1(]0,T[ \times \Rn)$ to $\rho_i$ and $\left(\nabla(V_i[\grho])\cdot \nabla \phi_i\right) \in L^{\infty} ([0,T] \times \Rn)$, then
$$ J_3 \rightarrow 0 \text{ as } h\rightarrow 0.$$

\item For $ J_2$,  we use the fact that $\nabla \phi_i$ is a Lipschitz function and that $\nabla(V_i[\grho])$ is bounded thanks to \eqref{Hyp 2 V}, and then,
$$ J_2 \leqslant CT \|D^2 \phi_i \|_{L^{\infty} ([0,T] \times \Rn)} h \rightarrow 0.$$

\item Using assumption \eqref{Hyp 3 V}, we have
\begin{eqnarray*}
J_1 &\leqslant & C \| \nabla \phi_i \|_{L^{\infty} ([0,T] \times \Rn)} \sum_{k=0}^{N-1}\int_{t_k}^{t_{k+1}}  W_2(\grho_{h}^{k},\grho(t,\cdot)) \,dt\\
&\leqslant & C \| \nabla \phi_i \|_{L^{\infty} ([0,T] \times \Rn)} \sum_{k=0}^{N-1}\int_{t_k}^{t_{k+1}}(W_2(\grho_{h}^{k},\grho_{h}^{k+1})+W_2(\grho_{h}^{k+1},\grho(t,\cdot))) \,dt\\
&\leqslant & C \| \nabla \phi_i \|_{L^{\infty} ([0,T] \times \Rn)} \left( h\sum_{k=0}^{N-1}W_2(\grho_{h}^{k},\grho_{h}^{k+1}) + \sum_{k=0}^{N-1}\int_{t_k}^{t_{k+1}} W_2(\grho_{h}^{k+1},\grho(t,\cdot)) \,dt \right)\\
&\leqslant & C \| \nabla \phi_i \|_{L^{\infty} ([0,T] \times \Rn)} \left( T\sum_{k=0}^{N-1}W_2^2(\grho_{h}^{k},\grho_{h}^{k+1}) + \int_0^T W_2(\grho_{h}(t,\cdot),\grho(t,\cdot)) \,dt  \right)
\end{eqnarray*}
\end{itemize}
According to \eqref{estimation distance}, we obtain
$$ T\sum_{k=0}^{N-1}W_2^2(\grho_{h}^{k},\grho_{h}^{k+1}) \leqslant lCT h \rightarrow 0$$
when $h\searrow0$. Moreover,
$$\int_0^T W_2(\grho_{h}(t,\cdot),\grho(t,\cdot)) \,dt \leqslant T\sup_{t\in [0,T]} W_2(\grho_{h}(t,\cdot),\grho(t,\cdot))\rightarrow 0,$$
when $h$ goes to $0$, which proves that
$$ J_1 \rightarrow 0 \text{ as } h\rightarrow 0.$$

\end{itemize}

If we combine all these convergences, theorem \ref{existence Rn} is proved.
\end{proof}

\section{The case of a bounded domain $\Omega$.}\label{borné}

In this section, we work on a smooth bounded domain $\Omega$ of $\Rn$ and only with one density but, as in the whole space, the result readily extends to systems. Our aim is to solve \eqref{systeme}. We remark that $\Omega$ is not taken convex so we can not use the flow interchange argument anymore because this argument uses the displacement convexity of the Entropy. Moreover since $\Omega$ is bounded, the solution has to satisfy some boundary conditions contrary to the periodic case \cite{CL} or in $\Rn$.  In our case, we study \eqref{systeme} with no flux boundary condition, which is the natural boundary condition for gradient flows, i.e we want to solve
\begin{eqnarray}
\label{equation borné}
 \left\{\begin{array}{ll}
 \partial_t \rho -  \dive(\rho \nabla (V[\rho])) -\Delta P(\rho) =0 & \text{   on } \mathbb{R}^+ \times \Omega,\\
(\rho \nabla (V[\rho])+\nabla P(\rho))\cdot \nu =0 &\text{   on } \R^+ \times \partial \Omega,\\
\rho(0, \cdot)= \rho_{0} &\text{   on } \Rn,\\
\end{array}\right.
\end{eqnarray}
where $\nu$ is the outward unit normal to $\partial \Omega$. 

\smallskip
We say that $\rho \, : \, [0, +\infty[ \rightarrow \mathcal{P}^{ac}(\Omega)$ is a weak solutions of \eqref{equation borné}, with $F \in \h_m$, if $\rho \in \mathcal{C}([0,+\infty[;\mathcal{P}^{ac}(\Omega)) \cap L^{m}(]0,T[ \times \Omega)$, $P(\rho) \in L^1(]0,T[\times \Omega)$, $\nabla P(\rho) \in \mathcal{M}^n([0,T]\times \Omega)$ for all $T<\infty$ and if for all $\varphi \in \mathcal{C}^\infty_c([0,+\infty[ \times \Rn)$, we have
$$ \int_0^\infty \int_\Omega \left[(\partial_t \varphi -\nabla \varphi \cdot \nabla (V[\rho])) \rho  -\nabla P(\rho) \cdot \nabla \varphi\right] = -\int_\Omega \varphi(0,x) \rho_{0}(x).$$
Since test functions are in $\mathcal{C}^\infty_c([0,+\infty[ \times \Rn)$, we do not impose that they vanish on the boundary of $\Omega$, which give Neumann boundary condition.\\

\begin{thm}
\label{existence solution sur borné}
Let $F \in \h_m$ for $m\geqslant 1$ and let $V$ satisfies \eqref{Hyp 1 V}, \eqref{Hyp 2 V}, \eqref{Hyp 3 V}.
If we assume that $\rho_0 \in \mathcal{P}^{ac}(\Omega)$ satisfies
\begin{align}
\label{hyp energie initiale borné}
\mathcal{F}(\rho_{0})+\mathcal{V}(\rho_{0}|\rho_{0}) < +\infty,
\end{align}
with
$$ \F(\rho) :=\left\{\begin{array}{ll}
\int_{\Omega} F(\rho(x))\,dx & \text{ if } \rho \ll \mathcal{L}^n_{\vert \Omega},\\
+\infty & \text{ otherwise,}
\end{array}\right. \text{  and  } \V(\rho| \mu):= \int_{\Omega} \V[\mu] \rho \,dx.$$
then \eqref{equation borné} admits at least one weak solution.
\end{thm}

The proof of this theorem is different from the one on $\Rn$ because we will not use the flow interchange argument of Matthes, McCann and Savaré to find strong convergence since $\Omega$ is not assumed convex. First, we will find an a.e equality using the first variation of energies in order to have a discrete equation, as in proposition \ref{equation discrete}. Then, we will derive an new estimate on the gradient of some power of $\rho_h$  from this a.e equality. To conclude, we will use again the refined version of Aubin-Lions lemma of Rossi and Savaré in \cite{RS}.

On $\Omega$ we can define, with the semi-implicit JKO scheme, the sequence $(\rho_{h}^k)_k$ but this time we minimize 
$$ \rho \mapsto \E_h(\rho|\rho_h^{k-1}):=\frac{1}{2h}W_2^2(\rho,\rho_h^{k-1}) + \F(\rho) + \V(\rho|\rho_h^{k-1})$$
on $\mathcal{P}(\Omega)$. The proof of existence and uniqueness of $\rho_{h}^k$ is the same as in proposition \ref{existence et unicité minimiseur}. It is even easier because on a bounded domain $\F$ is bounded from below for all $m\geqslant 1$. We find also the same estimates than in the proposition \ref{estimates} on the functional and the distance (see for example \cite{A},\cite{CL}).\\

Now we will establish a discrete equation satisfied by the piecewise interpolation of the sequence $(\rho_{h}^k)_k$ defined by, for all $k\in \mathbb{N}$,
 $$ \rho_{h}(t)= \rho_{h}^k \text{ if } t\in ((k-1)h,kh].$$

\smallskip
\begin{prop}
For every $k\geqslant 0$, we have 
\begin{eqnarray}
\label{a.e equality}
(y-T_k(y))\rho_{h}^{k+1} +h\nabla(V[\rho_h^{k}])\rho_{h}^{k+1}+h\nabla (P(\rho_{h}^{k+1}))=0 \; \; a.e \text{ on } \Omega,
\end{eqnarray}
where $T_k$ is the optimal transport map between $\rho_{h}^{k+1}$ and $\rho_{h}^{k}$. Then $\rho_h$ satisfies
\begin{eqnarray}
\int_0^T \int_{\Omega} \rho_{h}(t,x)  \partial_t \varphi(t,x) \,dxdt &=& h\sum_{k=0}^{N-1}\int_{\Omega}  \nabla (V [\rho_{h}^{k}])(x)  \cdot \nabla \varphi (t_k,x)\rho_{h}^{k+1}(x) \,dxdt \nonumber\\
&&  +h\sum_{k=0}^{N-1} \int_{\Omega} \nabla P(\rho_{h}^{k+1}(x)) \cdot  \nabla \varphi (t_k,x)  \, dx \label{equation discrete borné}\\
&& +\sum_{k=0}^{N-1} \int_{\Omega \times \Omega} \mathcal{R}[\varphi(t_k,\cdot)](x,y) d\gamma_k (x,y)\nonumber\\
&& - \int_{\Omega} \rho_{0}(x) \varphi(0,x) \, dx,\nonumber
\end{eqnarray}
with $N= \left\lceil\frac{T}{h} \right\rceil$, for all $\phi \in \mathcal{C}^\infty_c([0,T) \times \Rn)$, $\gamma_k$ is the optimal transport plan in $\Gamma (\rho_{h}^{k},\rho_{h}^{k+1})$ and
$$ |\mathcal{R}[\phi](x,y)| \leqslant \frac{1}{2} \|D^2 \phi \|_{L^\infty (\R \times \Rn)} |x- y|^2.$$

\end{prop}

\begin{proof}
First, we prove the equality \eqref{a.e equality}. As in proposition \ref{equation discrete}, taking the first vartiation in the semi-implicit JKO scheme, we find
for all $\xi \in \mathcal{C}_c^\infty(\Omega;\Rn)$,
\begin{equation}
\label{première variation borné}
\int_{\Omega} (y-T_k(y))\cdot \xi(y) \rho_{h}^{k+1}(y)\,dy+h\int_{\Omega} \nabla(V[\rho_h^{k}])\cdot \xi \rho_{h}^{k+1} -h\int_\Omega P(\rho_{h}^{k+1}) \dive(\xi) =0,
\end{equation}
where $T_k$ is the optimal transport map between $\rho_{h}^{k+1}$ and $\rho_{h}^{k}$. Now we claim that $P(\rho_{h}^{k+1}) \in W^{1,1}(\Omega)$. Indeed, since $F$ controls $x^m$ and $P$ is controlled by $x^m$ then \eqref{estimation entropie} gives $P(\rho_{h}^{k+1}) \in L^{1}(\Omega)$. Moreover, \eqref{première variation borné} gives 

\begin{eqnarray*}
 \left| \int_\Omega P(\rho_{h}^{k+1}) \dive(\xi) \right| \leqslant \left[ \int_{\Omega} \frac{|y-T_k(y)|}{h}\rho_h^{k+1} +C \right]  \|\xi \|_{L^{\infty}(\Omega)} \leqslant \left[ \frac{W_2(\rho_h^k,\rho_h^{k+1})}{h} +C\right] \|\xi\|_{L^{\infty}(\Omega)}.
\end{eqnarray*}
This implies $P(\rho_h^{k+1}) \in BV(\Omega)$ and $\nabla P(\rho_h^{k+1})= \nabla V[\rho_h^{k}] \rho_h^{k+1} + \frac{Id-T_k}{h}\rho_h^{k+1}$ in $\mathcal{M}^n(\Omega)$. And, since $ \nabla V[\rho_h^{k}] \rho_h^{k+1} + \frac{Id-T_k}{h}\rho_h^{k+1} \in L^1(\Omega)$, we have $P(\rho_h^{k+1})\in W^{1,1}(\Omega)$ and \eqref{a.e equality}.

\smallskip

Now, we verify that $\rho_h$ statisfies \eqref{equation discrete borné}. We start to take the scalar product between \eqref{a.e equality} and $\nabla \varphi$ with $\varphi \in \mathcal{C}^\infty_c([0,T)\times \Rn)$, and we find, for all $t\in[0,T)$,
\begin{equation}
\label{prem var apre IPP}
\int_{\Omega} (y-T_k(y))\cdot \nabla \varphi(t,y) \rho_{h}^{k+1}(y)\,dy+h\int_{\Omega} \nabla(V[\rho_h^{k}])(y)\cdot \nabla \varphi(t,y) \rho_{h}^{k+1}(y) \,dy +h\int_\Omega \nabla(P(\rho_{h}^{k+1}))(y)\cdot \nabla \varphi(t,y) \,dy =0.
\end{equation}

Moreover, if we extend $\varphi$ by $\varphi(0,\cdot)$ on $[-h,0)$, then

\begin{eqnarray*}
\int_0^T \int_{\Omega} \rho_{h}(t,x)  \partial_t \varphi(t,x)\,dxdt &=& \sum_{k=0}^{N} \int_{t_{k-1}}^{t_{k}} \int_{\Omega} \rho_{h}^k(x)  \partial_t \varphi(t,x)\,dxdt\\
&=& \sum_{k=0}^{N} \int_{\Omega} \rho_{h}^k(x) (\varphi(t_{k},x)-\varphi(t_{k-1},x))\,dx\\
&=& \sum_{k=0}^{N-1} \int_{\Omega} \varphi(t_{k},x) (\rho_{h}^{k}(x) -\rho_{h}^{k+1}(x))\,dx - \int_{\Rn} \rho_{0}(x) \varphi(0,x) \, dx.
\end{eqnarray*}

And using the second order Taylor-Lagrange formula, we find
$$\int_{\Omega\times \Omega} (\varphi(kh,x)- \varphi(kh,y))\,d \gamma_k(x,y)=\int_{\Omega\times \Omega} \nabla \varphi(kh,y)\cdot (x-y)\, d\gamma_k(x,y)+\int_{\Omega \times \Omega} \mathcal{R}[\varphi(t_k,\cdot)](x,y) d\gamma_k (x,y).$$
This concludes the proof if we sum on $k$ and use \eqref{prem var apre IPP}.

\end{proof}

\begin{rem}
\label{rem a.e equality}
We remark that equality \eqref{a.e equality} is still true in $\Rn$. Indeed, the first part of the proof does not depend of the domain and we can use this argument on $\Rn$. This equality will be used in section \ref{uniqueness} to obtain uniqueness result.
\end{rem}

In the next proposition, we propose an alternative argument to the flow interchange argument to get an estimate on the gradient of $\rho_h$. Differences with the flow interchange argument are that we do not need to assume the space convexity and boundary condition on $\nabla V[\rho]$. Moreover we do not obtain exactly the same estimate. Indeed, in proposition \ref{estimation gradient}, $\nabla \rho_h^{m/2}$ is bounded in $L^2((0,T)\times \Rn)$ whereas in the following proposition we establish a bound on $\nabla \rho_h^{m}$ in $L^1((0,T)\times \Rn)$ using \eqref{a.e equality}.

\begin{prop}
\label{estimation gradient borné W11}
There exists a constant $C$ which does not depend on $h$ such that 
$$ \|\rho_{h}^{m}\|_{L^1([0,T];W^{1,1}(\Omega))} \leqslant CT$$
for all $T>0$. 
\end{prop}
\begin{proof}
According to \eqref{a.e equality}, we have
\begin{eqnarray*}
h\int_{\Omega} |\nabla(P(\rho_h^{k+1}))| \,dx & \leqslant & W_2(\rho_h^k,\rho_h^{k+1}) +hC.
\end{eqnarray*}
Then if we sum on $k$ from $0$ to $N-1$, we get
\begin{eqnarray*}
\int_0^T \int_\Omega |\nabla(P(\rho_h))|\,dxdt &\leqslant & \sum_{k=0}^{N-1} W_2(\rho_h^k,\rho_h^{k+1}) +TC\\
&\leqslant & N\sum_{k=0}^{N-1} W_2^2(\rho_h^k,\rho_h^{k+1}) +TC\\
&\leqslant & CT,
\end{eqnarray*} 
because of \eqref{estimation distance}. 

If $F(x)=x\log(x)$ then $P'(x)=1$ and if $F$ satisfies \eqref{hyp F}, then $F''(x) \geqslant Cx^{m-2}$ and $P'(x)=xF''(x) \geqslant Cx^{m-1}$. In both cases, we have $P'(x)\geqslant Cx^{m-1}$ (with $m=1$ for $x\log(x)$). So 
$$ \int_0^T \int_\Omega |\nabla(P(\rho_h))|\,dxdt = \int_0^T \int_\Omega P'(\rho_h)|\nabla\rho_h|\,dxdt \geqslant C\int_0^T \int_\Omega \rho_h^{m-1}|\nabla\rho_h|\,dxdt =C\int_0^T \int_\Omega |\nabla\rho_h^m|\,dxdt,$$ 

Which proves the proposition.
\end{proof}

Now we introduce $\G \, : \, L^m(\Omega) \rightarrow [0,+\infty]$ defined by
$$ \G(\rho):= \left\{\begin{array}{ll}
\| \rho^m \|_{BV(\Omega)}  & \text{ if } \rho \in \mathcal{P}^{ac}(\Omega) \text{ and } \rho^m \in BV(\Omega),\\
+\infty & \text{ otherwise.}
\end{array}\right.$$

\begin{prop}
$\G$ is lower semi-continuous on $L^m(\Omega)$ and its sublevels are relatively compact in $L^m(\Omega)$.
\end{prop}

\begin{proof}
First we show that $\G$ is lower semi-continuous on $L^m(\Omega)$. Let $\rho_n$ be a sequence which converges strongly to $\rho$ in $L^m(\Omega)$ with $\sup_{n} \G(\rho_n) \leqslant C<+\infty$. Without loss of generality, we assume that $\rho_n$ converges to $\rho$ a.e.\\
Since $C<+\infty$, the functions $\rho_n^m$ are uniformly bounded in $BV(\Omega)$. So we know that $\rho_n^m$ converges weakly in $BV(\Omega)$ to $\mu$. But since $\Omega$ is smooth and bounded, the injection of $BV(\Omega)$ into $L^1(\Omega)$ is compact. We can deduce that $\mu=\rho^m$ and $\rho_n^m$ converges to $\rho^m$ strongly in $L^1(\Omega)$. Then by lower semi-continuity of the $BV$-norm in $L^1$, we obtain
$$ \G(\rho) \leqslant \liminf_{n\nearrow +\infty} \G(\rho_n).$$

Now, we have to prove that the sublevels, $A_c:=\{ \rho \in L^m(\Omega) \, : \, \G(\rho) \leqslant c\}$, are relatively compact in $L^m(\Omega)$. Since $i \, : \, \eta \in L^1(\Omega) \mapsto \eta^{1/m} \in L^m(\Omega)$ is continuous, we just have to prove that $B_c:=\{\eta=\rho^m \, : \, \rho \in A_c \}$ is relatively compact in $L^1(\Omega)$. So to conclude the proof, it is enough to notice that $B_c$ is a bounded subset of $BV(\Omega)$ and that the injection of $BV(\Omega)$ into $L^1(\Omega)$ is compact.
\end{proof}

Now we can apply Rossi-Savaré theorem (theorem \ref{Rossi-Savaré}) to have the strong convergence in $L^{m}( ]0,T[ \times \Omega)$ of $\rho_{h}$ to $\rho$ and then we find the strong convergence in $L^1(]0,T[ \times \Omega)$ of $P(\rho_{h}) $ to $P(\rho)$, for all $T>0$, using the fact that $P$ is controlled by $x^m$ \eqref{hyp F} and Krasnoselskii theorem (see \cite{DF}, chapter 2).

Moreover, since 
$$\int_0^T \int_\Omega |\nabla(P(\rho_h))|\,dxdt \leqslant  CT,$$
we have 
\begin{equation}
\label{cv measure}
\nabla (P(\rho_h))dxdt \rightharpoonup \mu \text{  in } \mathcal{M}^n([0,T]\times \Omega),
\end{equation} 
i.e 
$$ \int_0^T\int_\Omega \xi \cdot  \nabla (P(\rho_h))dxdt \rightarrow \int_0^T\int_\Omega \xi \cdot \, d\mu,$$
for all $\xi \in \mathcal{C}_b([0,T]\times \Omega)$ (this means that we do not require $\xi$ to vanish on $\partial \Omega$).
But since $P(\rho_h)$ converges strongly to $P(\rho)$ in $L^1([0,T]\times \Omega)$, $\mu=\nabla (P(\rho))$.

To conclude, we pass to the limit in \eqref{equation discrete borné} and theorem \ref{existence solution sur borné} follows.
 
\section{Uniqueness of solutions.}\label{uniqueness}

In this section, we prove uniqueness result if $\Omega$ is a convex set. The convexity assumption is important because uniqueness arises from a displacement convexity argument.

Without loss of generality, we focus here on internal energy defined on the subset of probability densities with finite second moment $\mathcal{P}_{2}^{ac}(\Omega)$ and given by
$$ \forall \rho \in \mathcal{P}_{2}^{ac}(\Omega), \qquad \F(\rho):= \int_\Omega F(\rho(x)) \, dx, $$
with $F \, : \, \R^+ \rightarrow \R$ a convex function of class $\mathcal{C}^2((0,+\infty))$ with $F(0)=0$. We recall that for all $\rho,\mu$ in $\mathcal{P}_{2}^{ac}$, there exists (see for example \cite{B},\cite{S},\cite{V1}) a unique optimal transport map $T$ between $\rho$ and $\mu$ such that
$$ W_2^2(\rho,\mu)=\int_{\Omega \times \Omega} |x-T(x)|^2\rho(x) \, dx.$$ 
The McCann's interpolation is defined by $T_t:=Id+t(T-Id)$ for any $t\in [0,1]$. Then the curve $t\in [0,1]  \mapsto \rho_t$, with $\rho_t:={T_t}_{\#}\rho$, is the Wasserstein geodesic between $\rho$ and $\mu$ (\cite{S},\cite{V1},\cite{AGS}).
 
An internal energy $\F$ is said displacement convex if 
$$ t\in[0,1] \mapsto \F(\rho_t) \text{ is convex.}$$

Moreover, we say that $F \, :\, [0,+\infty) \rightarrow \R$ satisfy McCann's condition if
\begin{equation}
\label{McCann condition}
x \in (0,+\infty) \mapsto x^{n}F(x^{-n}) \text{ is convex nonincreasing}.
\end{equation}
McCann showed in \cite{MC} that if $F$ satisfy \eqref{McCann condition}, then $\F$ is displacement convex.

\smallskip
Now we will state a general uniqueness argument based on geodesic convexity. This result has been already proved in \cite{CL} in the flat-torus case and the proof is the same in our case.

\begin{thm}
Assume that $V_i$ satisfy \eqref{Hyp 2 V} and \eqref{Hyp 3 V} and $F_{i} \in \h_{m_i}$ satisfy \eqref{McCann condition}. Let $\boldsymbol{\rho}^1:=(\rho_1^1, \dots, \rho_l^1)$ and $\boldsymbol{\rho}^2:=(\rho_1^2, \dots, \rho_l^2)$ two weak solutions of \eqref{systeme} or \eqref{existence solution sur borné} with initial conditions $\rho_i^1(0, \cdot)= \rho_{i,0}^1$ and $\rho_i^2(0, \cdot)= \rho_{i,0}^2$. If for all $T<+\infty$,
\begin{equation}
\label{hyp gradient field}
\int_0^T \sum_{i=1}^l \|v^1_{i,t}\|_{L^2(\rho_{i,t}^1)} \,dt + \int_0^T \sum_{i=1}^l \|v^2_{i,t}\|_{L^2(\rho_{i,t}^2)} \,dt <+\infty,
\end{equation}
with, for $j\in \{1,2\}$,
 $$ v^j_{i,t}:= - \frac{\nabla P_i(\rho_{i,t}^j)}{\rho_{i,t}^j} -\nabla V_i[\grho_{t}^j],$$

then for every $t\in   [0 ,T ]$,
$$  W_2^2(\boldsymbol{\rho}_t^1,\boldsymbol{\rho}_t^2) \leqslant e^{4Ct}W_2^2(\boldsymbol{\rho}^1_0,\boldsymbol{\rho}^2_0).$$
 In particular, we have uniqueness for the Cauchy problems \eqref{systeme} and \eqref{existence solution sur borné}.
\end{thm}

In the following proposition, we will prove that assumption \eqref{hyp gradient field} holds if $\Omega$ is a smooth bounded convex subset of $\Rn$ or if $\Omega=\Rn$. 

\begin{prop}
\label{gradient field assumption}
Let $\boldsymbol{\rho}:=(\rho_1, \dots, \rho_l)$ be a weak solution of \eqref{systeme} obtained with the previous semi-implicit JKO scheme. Then $\rho_i$ satisfies \eqref{hyp gradient field} for all $i\in [\![ 1, l]\!]$.
\end{prop}

\begin{proof}We do not separate the cases where $\Omega$ is a bounded set or is $\Rn$. 
We split the proof in two parts. First, we show that \eqref{hyp gradient field} is satisfied by $\rho_{i,h}$ defined in \eqref{interpolation}. Then by a l.s.c argument we will conclude the proof.

\begin{itemize}
\item In the first step, we show that $\rho_{i,h}$ satisfies
\begin{equation}
\label{disceet gradient field assumption}
\int_0^T\int_\Omega | \nabla F_i'(\rho_{i,h}) + \nabla V_i[\grho_{h}]|^2 \rho_{i,h} \,dxdt \leqslant C,
\end{equation}
where $C$ does not depend of $h$.

By equality \eqref{a.e equality} and remark \ref{rem a.e equality}, we have
$$ \nabla F_i'(\rho_{i,h}^{k+1}) + \nabla V_i[\grho_{h}^k] = \frac{T_k(y)-y}{h} \qquad \rho_{i,h}^{k+1}-\text{a.e on } \Omega,$$
where $T_k$ is the optimal transport map between $\rho_{i,h}^{k+1}$ and $\rho_{i,h}^{k}$. Then if we take the square, multiply by $\rho_{i,h}^{k+1}$ and integrate on $\Omega$, we find
$$\int_\Omega | \nabla F_i'(\rho_{i,h}^{k+1})+ \nabla V_i[\grho_{h}^k]|^2 \rho_{i,h}^{k+1} \,dx \leqslant\frac{1}{h^2}W_2^2(\rho_{i,h}^{k+1},\rho_{i,h}^{k}).$$
Now using \eqref{Hyp 3 V}, we get
\begin{eqnarray*}
| \nabla F_i'(\rho_{i,h}^{k+1})+ \nabla V_i[\grho_{h}^{k+1}]| &\leqslant & | \nabla F_i'(\rho_{i,h}^{k+1})+ \nabla V_i[\grho_{h}^{k}]|+|\nabla V_i[\grho_{h}^{k}]- \nabla V_i[\grho_{h}^{k+1}]|\\
&\leqslant &| \nabla F_i'(\rho_{i,h}^{k+1})+ \nabla V_i[\grho_{h}^{k}]|+CW_2(\grho_{h}^{k+1},\grho_{h}^{k})
\end{eqnarray*}
So we have
$$ \int_\Omega | \nabla F_i'(\rho_{i,h}^{k+1})+ \nabla V_i[\grho_{h}^{k+1}]|^2 \rho_{i,h}^{k+1} \,dx \leqslant C\left(\frac{1}{h^2}W_2^2(\rho_{i,h}^{k+1},\rho_{i,h}^{k}) +W_2^2(\grho_{h}^{k+1},\grho_{h}^{k}) \right).$$
Then using \eqref{estimation distance}, we finally get
\begin{eqnarray*}
\int_0^T\int_\Omega  | \nabla F_i'(\rho_{i,h})+ \nabla V_i[\grho_{h}]|^2 \rho_{i,h} \,dxdt &=& h\sum_{k=0}^{N-1}\int_\Omega  | \nabla F_i'(\rho_{i,h}^{k+1})+ \nabla V_i[\grho_{h}^{k+1}]|^2 \rho_{i,h}^{k+1} \,dx\\
&\leqslant & C\left( \frac{1}{h}\sum_{k=0}^{N-1} W_2^2(\rho_{i,h}^{k+1},\rho_{i,h}^{k}) +1 \right)\\
&\leqslant & C.
\end{eqnarray*}

%
%
%
%
%

\item To conclude, we have to pass to the limit in \eqref{disceet gradient field assumption}. First, we claim that $\nabla P_i(\rho_{i,h})$ converges to $\nabla P_i(\rho_{i})$ in $\mathcal{M}^n([0,T]\times \Omega)$. In a bounded set, this has been proved in \eqref{cv measure}. In $\Rn$ thanks to the previous step, we have
\begin{eqnarray*}
\int_0^T \int_{\Rn} |\nabla P_i(\rho_{i,h})|dt &:=& \int_0^T \int_{\Rn} |\nabla F'_i(\rho_{i,h})|\rho_{i,h}\,dx dt\\
&\leqslant & \int_0^T\int_{\Rn} (|\nabla F_i'(\rho_{i,h}) |^2 +1)\rho_{i,h}\\
&\leqslant & C,
\end{eqnarray*}
which gives the result because $P_i(\rho_{i,h})$ strongly converges in $L^1([0,T]\times \Rn)$ to $P_i(\rho_{i})$.\\

Let $\psi \, : \, \R^{n+1} \rightarrow \R \cup \{+\infty\}$ defined by
$$ \psi(r,m):= \left\{ \begin{array}{ll}
\frac{|m|^2}{r} & \text{ if } (r,m)\in ]0,+\infty[ \times \Rn,\\
0 & \text{ if } (r,m)=(0,0),\\
+\infty & \text{ otrherwise,}
\end{array}\right.$$
as in \cite{BB}. And define $\Psi \, : \, \mathcal{M}((0,T) \times \Omega)\times \mathcal{M}^n((0,T)\times \Omega)  \rightarrow \R \cup \{+\infty\}$, as in \cite{BJO}, by
$$ \Psi(\rho,E):=\left\{\begin{array}{ll}
\int_0^T \int_{\Omega} \psi(d\rho/d\mathcal{L},dE/d\mathcal{L}) \,dxdt & \text{ if } \rho \geqslant 0,\\
+\infty & \text{ otrherwise,}
\end{array} \right.$$
where $d\sigma/d\mathcal{L}$ is Radon-Nikodym derivative of $\sigma$ with respect to $\mathcal{L}_{|[0,T]\times \Omega}$. We can remark that since $\psi(0,m)=+\infty$ for any $m\neq 0$, we have
$$ \Psi(\rho,E) <+\infty \Rightarrow E \ll \rho.$$

With this definition, we can rewrite \eqref{disceet gradient field assumption} as
$$ \Psi(\rho_{i,h},\nabla P_i(\rho_{i,h})+\nabla V_i[\grho_h]\rho_{i,h})=\int_0^T\int_\Omega | \nabla F_i'(\rho_{i,h}) + \nabla V_i[\grho_{h}]|^2 \rho_{i,h} \,dxdt \leqslant C,$$
which, in particular, implies that $\nabla P_i(\rho_{i,h}) \ll\rho_{i,h} \ll \mathcal{L}_{|[0,T]\times \Omega}$.

Moreover, according to \cite{BBu}, $\Psi$ is lower semicontinuous on $\mathcal{M}([0,T]\times \Omega) \times \mathcal{M}^n([0,T]\times \Omega)$. So, it holds
$$ \Psi(\rho_i,\nabla P_i(\rho_{i})+\nabla V_i[\grho]\rho_{i}) \leqslant \liminf_{h\searrow 0}\Psi(\rho_{i,h},\nabla P_i(\rho_{i,h})+\nabla V_i[\grho_h]\rho_{i,h}) \leqslant C,$$
which imply $\nabla P_i(\rho_{i}) \ll \rho_i \ll \mathcal{L}_{|[0,T]\times \Omega}$ and conclude the proof because
\begin{eqnarray*}
\int_0^T \int_{\Omega} \left|\frac{\nabla P_i(\rho_i)}{\rho_i}+\nabla V_i[\grho]\right|^2 \rho_i \,dxdt &=& \int_0^T \int_{\Omega} \frac{|\nabla P_i(\rho_{i})+\nabla V_i[\grho]\rho_{i}|^2}{\rho_i} \,dxdt\\
&=& \Psi(\rho_i,\nabla P_i(\rho_{i})+\nabla V_i[\grho]\rho_{i}) \\
&\leqslant & C.
\end{eqnarray*}

\end{itemize}
\end{proof}

{\textbf{Acknowledgements:}} The author wants to gratefully thank G. Carlier for his help and advices.

\end{document}